\documentclass[a4paper,reqno]{amsart}
\usepackage{tikz}
\usepackage{xcolor}
\usepackage{enumitem}
\usetikzlibrary{patterns}
\newcommand\R{\mathbb R}

\newcommand\cC{\mathcal C}
\newcommand\cD{\mathcal D}

\begin{document}
\title[Pressureless gas dynamics]
{Initial-boundary value problem for 1D pressureless gas dynamics}

\author[L. Neumann, M. Oberguggenberger, M. R. Sahoo, A. Sen]
{Lukas Neumann, Michael Oberguggenberger, Manas R. Sahoo, Abhrojyoti Sen}

\address{\flushleft \parbox{12cm}{Lukas Neumann and Michael Oberguggenberger \newline
Department for Basic Sciences in Engineering Sciences\\
University of Innsbruck, Technikerstra{\ss}e 13, 6020 Innsbruck, Austria.\\
\rm\textit{Email address}: \texttt{Lukas.Neumann@uibk.ac.at, Michael.Oberguggenberger@uibk.ac.at}}}
%\email{Lukas.Neumann@uibk.ac.at, Michael.Oberguggenberger@uibk.ac.at}
\address{\flushleft \parbox{12cm}{Manas R. Sahoo and Abhrojyoti Sen \newline
School of Mathematical Sciences\\
National Institute of Science Education and Research, HBNI\\
Jatni, Kurda, Bhubaneswar 752050, India.\\
\rm\textit{Email address}: \texttt{manas@niser.ac.in, abhrojyoti.sen@niser.ac.in}}}
%\email{manas@niser.ac.in, abhrojyoti.sen@niser.ac.in}

\begin{abstract}
The paper considers the system of pressureless gas dynamics in one space dimension. The question of solvability of the initial-boundary value problem is addressed. Using the method of generalized potentials and characteristic triangles, extended to the boundary value case, an explicit way of constructing measure-valued solutions is presented. The prescription of boundary data is shown to depend on the behavior of the generalized potentials at the boundary. We show that  the constructed solution satisfies an entropy condition and it conserves mass, whereby mass may accumulate at the boundary. Conservation of momentum again depends on the behavior of the generalized boundary potentials. There is a large amount of literature where the initial value problem for the pressureless gas dynamics model has been studied. To our knowledge, this paper is the first one which considers the initial-boundary value problem.
\end{abstract}

\maketitle
\numberwithin{equation}{section}
\numberwithin{equation}{section}
\newtheorem{theorem}{Theorem}[section]
\newtheorem{remark}[theorem]{Remark}
\newtheorem{lem}[theorem]{Lemma}
\newtheorem{definition}[theorem]{Definition}
\newtheorem{corollary}[theorem]{Corollary}

\section{Introduction}
This paper addresses the solvability of the initial-boundary value problem for the system of pressureless gas dynamics
\begin{equation}
\begin{aligned}
&\rho_t+(\rho u)_x=0\\
&(\rho u)_t+(\rho u^2)_x=0
\end{aligned}
\label{e1.1}
\end{equation}
in one space dimension. Here $\rho$ denotes the density and $u$ the velocity. The two lines in equation \eqref{e1.1} express conservation of mass and of momentum, respectively. We adjoin initial data
\begin{equation}
\rho(x,0) = \rho_0(x),\quad u(x,0) =  u_0(x),\quad x > 0,
\label{e1.2}
\end{equation}
and ask under what conditions and in what sense boundary data
\begin{equation}\label{e1.3}
u(0,t)=u_b(t),\quad (\rho u) (0,t) = (\rho_b u_b)(t), \quad t>0,
\end{equation}
can be prescribed. It is assumed that the data $u_0$ and $u_b$ are bounded measurable functions with $u_b > 0$. Further, $\rho_0$ and $\rho_b $ are positive locally bounded measurable functions.

The initial value problem \eqref{e1.1}, \eqref{e1.2} has been intensively studied in the literature. The key issue is that, in general, $\rho$ is no longer a function, but a measure. This led to the introduction of various strongly related notions of weak solutions, such as measure solutions \cite{Bouchut94}, duality solutions \cite{BouchutJames99} (based on \cite{BouchutJames98}), duality solutions obtained by vanishing viscosity \cite{Boudin00},
mass and momentum potentials \cite{BrenierGrenier98,d3,Huang05}, together with generalized characteristics \cite{Wang97},
generalized potentials and variational principles \cite{ERykovSinai96,Wang01,WHD97}. In this paper, we shall extend the approach of \cite{Wang01,WHD97} to the boundary value problem.

Let us begin by discussing what is meant by a generalized solution to the system of differential equations \eqref{e1.1}. We shall construct locally bounded measurable functions $m(x,t)$, $u(x,t)$ such that for almost all $t$, $m(x,t)$ is of locally bounded variation with respect to  $x$. Thus, for almost all $t$, the distributional derivative $m_x$ defines a Radon measure $\rho$. In addition, $u$ is measurable with respect to $\rho$. Following \cite[Definition 1.1]{WHD97}, the pair $(\rho,u)$ is viewed as a generalized solution to \eqref{e1.1}, if
\begin{equation}\label{system_m}
\begin{array}{lcr}
    \displaystyle\iint \varphi_t m\, dx\, dt - \iint \varphi u\, dm\, dt = 0\\[8pt]
    \displaystyle\iint \big(\psi_t u + \psi_x u^2\big)\, dm\, dt = 0
\end{array}
\end{equation}
for all test functions $\varphi,\psi\in \cD(\mathbb{R}^2_+)$. The construction of the solution $(m,u)$ to \eqref{system_m} will be based on the method of generalized potentials and characteristic triangles from \cite{WHD97}. However, differently from \cite{WHD97}, we will need two types of generalized potentials (initial and boundary potential) and their relation, as well as different types of characteristic triangles, depending on the location of their apex.

More precisely, the initial and boundary potentials are defined by
\begin{align}
&F(y,x,t)=\int_{0}^{y}[tu_0(\eta)+\eta-x]\rho_0(\eta)d\eta,\label{e2.1}\\
&G(\tau, x, t)=\int_{0}^{\tau} [x-u_b(\eta)(t-\eta)]\rho_b(\eta)u_b(\eta)d\eta\,.
\label{e2.2}
\end{align}
Further,
\begin{equation}
F(x,t)=\min_{y\in [0, \infty)}F(y,x,t),
\label{e2.6}
\end{equation}
\begin{equation}
G(x,t)=\displaystyle{\min_{{\tau\in[0,\infty)}}}G(\tau, x, t).
\label{e2.5}
\end{equation}
The characteristic triangles with apex $(x,t)$ will depend on whether $F(x,t) < G(x,t)$, $F(x,t) > G(x,t)$ or $F(x,t) = G(x,t)$. For $x=0$, the respective relation between $F(0,t)$ and $G(0,t)$ will also decide about the assumption of the boundary data.

To further clarify the solution concept, we wish to show that $(\rho,u)$ actually is a weak solution to system \eqref{e1.1} in its proper sense.
Let us recall the measure theoretic point of view and the distributional point of view (for simplicity in the one-dimensional case). If $m$ is a function of locally bounded variation, it defines a Lebesgue-Stieltjes measure $dm$. On the other hand, its derivative in the sense of distributions defines a Radon measure $\rho = m_x$. The two objects are the same, identified by the chain of equalities
\[
  \int_\mathbb{R}\varphi(x)m(dx) = \langle\rho,\varphi\rangle = - \langle m,\varphi_x\rangle = -\int_\mathbb{R} \varphi_x(x)m(x) dx
\]
for $\varphi\in \cD(\mathbb{R})$. The first equality can be extended to $\varphi\in\cC(\mathbb{R})$ with compact support. What is more, the Lebesgue-Stieltjes integral can be extended to all functions $\varphi$ which are integrable with respect to $\rho$. This a priori makes no sense at the other end of the chain of equalities, but allows us to \emph{define} the product of the measure $\rho$ with the bounded, $\rho$-measurable function $u$ as the distribution given by
\[
   \langle \rho u,\varphi\rangle = \int_\mathbb{R} \varphi(x)u(x)m(dx)
\]
for $\varphi\in \cD(\mathbb{R})$.

Using this identification, the second line in \eqref{system_m} means
\[
  0 = \iint \big(\psi_t u + \psi_x u^2\big)\, dm\, dt = \langle \rho u,\psi_t\rangle + \langle \rho u^2,\psi_x\rangle,
\]
which is exactly the distributional meaning of the second line of \eqref{e1.1}. To obtain the first line of \eqref{e1.1}, one has to insert $\varphi_x$ in place of $\varphi$
in \eqref{system_m} to obtain
\[
   0 = \iint \varphi_{xt} m\, dx\, dt - \iint \varphi_x u\, dm\, dt = - \langle \rho, \varphi_t\rangle - \langle \rho u, \varphi_x\rangle.
\]
The actual proof of \eqref{system_m} will be done on yet a higher level. Apart from the mass potential $m(x,t)$, momentum and energy potentials $q(x,t)$ and $E(x,t)$ will be constructed, both bounded measurable functions which are in addition of bounded variation in $x$ for almost all $t$. Further, the Lebesge-Stieltjes measures $dq$ and $dE$ are absolutely continuous with respect to $dm$, namely
\[
  dq = u dm,\quad dE = \tfrac12 u^2 dm,
\]
and they satisfy the system
\begin{equation}\label{System_mqE}
\begin{array}{lcr}
    m_t + q_x = 0\\[4pt]
    q_t + (2E)_x = 0
\end{array}
\end{equation}
in the sense of distributions. By similar arguments as above, this system is equivalent with \eqref{system_m}.

What concerns the initial data, we will show that \eqref{e1.2} is satisfied in the sense that $\rho$ and $u$ are continuous functions of time with values in $\cD'(\R)$. Actually for almoust all $x$ we show  $\lim_{t\to 0} u(x,t)= u_0(x)$, and $\lim_{t\to 0}m(x,t) = \int_0^x \rho_0(y) dy$.

We turn to the assumption of the boundary data \eqref{e1.3}. As is well known from the theory of conservation laws \cite{bar, j3, JosephSahoo12,Le1}, one cannot arbitrarily prescribe boundary data, because a priori there is no control of the sign of $u(0+,t)$, except in the case when $u_0$ is positive and hence $u(x,t) > 0$ everywhere (recall that $u_b$ was assumed to be positive).

We will show the following: If $t>0$ is a Lebesgue point of $u_b$ and $\rho_b$ and $F(0,t) > G(0,t)$, then $\lim_{x\to 0+} u(x,t) = u_b(t)$. If in addition $u_b$ is continuously differentiable and $\rho_b$ is locally Lipschitz continuous, then $\lim_{x\to 0+} \rho(x,t)u(x,t) = \rho_b(t)u_b(t)$. If $F(0,t) < G(0,t)$ then $u(0+,t) < 0$ and the boundary condition \eqref{e1.3} cannot be fulfilled. Rather, it may happen that mass accumulates at the boundary in the form of $\delta\cdot (m(0+,t) - m(0,t))$. However, the solution we construct conserves total mass. Momentum is conserved at the points of time $t$ for which $F(0,t)\geq G(0,t)$, while it satisfies an inequality otherwise.
For further aspects of boundary conditions for systems involving measure solution, see \cite{ma2,NNOS17}.

The plan of exposition is as follows: Section 2 is devoted to the construction of the solution. In Section 3 it will be shown that the constructed solution satisfies system \eqref{system_m}, and hence \eqref{e1.1}. Section 4 addresses the assumption of initial and boundary values, as well as conservation of mass and momentum. In Section 5, it will be shown that the solution satisfies Oleinik's entropy condition. Finally, Section 6 contains a number of examples illustrating some of the possibly occurring effects.

\section{Construction of solution}
In this section, we construct the solution for the initial-boundary value problem. Recall the definition of the initial and boundary potentials \eqref{e2.1}, \eqref{e2.2},
\begin{align*}
&F(y,x,t)=\int_{0}^{y}[tu_0(\eta)+\eta-x]\rho_0(\eta)d\eta\,,\\
&G(\tau, x, t)=\int_{0}^{\tau} [x-u_b(\eta)(t-\eta)]\rho_b(\eta)u_b(\eta)d\eta\,.
\end{align*}

Given $(x,t)$, let $\tau^{*}(x,t)$ and $\tau_{*}(x,t)$ be the uppermost and lowermost points on the $t$-axis such that
\begin{equation*}
\min_{\tau \geq 0}G(\tau, x, t)= G(\tau^{*}(x,t), x, t)=G(\tau_{*}(x,t), x, t)\,.
\end{equation*}
Similarly, let $y_{*}(x,t)$ and $y^* (x,t)$ be the leftmost and rightmost points  respectively on the $x$-axis such that \begin{equation*}
\min_{y\geq 0}F(y,x,t)= F(y_{*}(x,t), x,t)= F(y^{*}(x,t), x,t)\,.
\end{equation*}
Note that these minima exist as real numbers, because $u_0$ is bounded from above and below and $u_b$ is positive. The following lemma collects some properties of the minimizers.
\begin{lem}\label{lnew}
With our assumptions on the initial and boundary data we have
\begin{enumerate}
\item $\tau_{*}(x,t)$ and $\tau^{*}(x,t)$ are, for fixed $x$, monotonically increasing in $t$ and for fixed $t$ monotonically decreasing in $x$. Moreover, we have for $t_1<t_2$ that $\tau^{*}(x,t_1)\leq\tau_{*}(x,t_2)$ and for $x_1<x_2$ that $\tau_{*}(x_1,t)\geq\tau^{*}(x_2,t)$.
\item $y_{*}(x,t)$ and $y^{*}(x,t)$ are, for fixed $t$, monotonically increasing in $x$ and for $x_1<x_2$ we have $y^{*}(x_1,t)\leq y_{*}(x_2,t)$.
\item $\tau_{*}(0,t)=\tau^{*}(0,t)=t$.
\item $y_{*}(x,t)$ is lower semicontinuous and $y^{*}(x,t)$ is upper semicontinuous.
\item $\tau_{*}(x,t)$ is lower semicontinuous and $\tau^{*}(x,t)$ is upper semicontinuous.
\end{enumerate}
\end{lem}
\begin{proof}
(1) Let $x,\,t_1,\,t_2>0$ be arbitrary but fixed and $\tau_1$ a minimizer of $G(\tau,x,t_1)$ and $\tau_2$ one of $G(\tau,x,t_2)$. Now we have
\[
0\leq G(\tau_2,x,t_1)-G(\tau_1,x,t_1)\,,\ \text{and } 0\leq G(\tau_1,x,t_2)-G(\tau_2,x,t_2)\,.
\]
Summing the two inequalities results in
\[
0 \leq (t_2-t_1)\int_{\tau_1}^{\tau_2}\rho_b(\eta){u_b}^2(\eta) d \eta\,.
\]
Since the term in the integral is positive by our assumptions we conclude that the minimizers have to be increasing in $t$. \\
Now on the other hand fixing $t,\,x_1,\,x_2$ and denoting by $\tau_1$ a minimizer of $G(\tau,x_1,t)$ and by $\tau_2$ one of $G(\tau,x_2,t)$ we derive in the same way
\[
0\leq (x_1-x_2)\int_{\tau_1}^{\tau_2}\rho_b(\eta)u_b(\eta) d \eta\,.
\]
From this one can conclude that the minimizers are decreasing in $x$.\\
(2) is Lemma 2.1 in \cite{WHD97}, from which also the proof of (1) is adopted. (3) is obvious, (4) see Lemma 2.2 in \cite{WHD97}.
(5) is proved along the lines of (4).
\end{proof}
\begin{remark}\label{rem:const}
If $\min_{\tau\geq 0} G(\tau,x,t)$ is constant on an interval $[x_1,x_2]\times \{t\}$ one can argue similarly to the proof of (1) above:
\[
0\leq G(\tau_2,x_1,t)-G(\tau_1,x_1,t)=G(\tau_2,x_1,t)-G(\tau_2,x_2,t)=(x_1-x_2)\!\!\int_0^{\tau_2}\!\!\!\!\!\rho_b(\eta)u_b(\eta) d \eta
\]
Now since $\rho_b$ and $u_b$ are assumed to be strictly positive we conclude that any minimizer $\tau_2$ of $G(\tau,x_2,t)$ has to be zero and thus $\min_{\tau\geq0} G(\tau,x_2,t)=0$. Since $G$ is constant on the interval it has to be equal to zero on the whole interval. Note also that the minimizers on the whole interval have to be zero (uniquely) because one can replace $x_2$ by any point in the interval in the estimate above.\\
This situation will correspond to the case when the solution contains a rarefaction wave starting at the origin.
\end{remark}
We first quote the following result, that was established by Wang, Huang and Ding \cite{WHD97} in their study of the initial value problem. This will be central also in our work for parts of the solution depending only on the initial data.
\begin{lem}\label{l1}
For fixed $(x,t)$, let the minimum ${\min_{y\in [0, \infty)}}F(y,x, t)$ be attained at $y(x,t)$. Then for any given point  $(x', t')$ on the line segment joining $(y(x,t), 0)$ and  $(x,t)$,
we have $F(y, x', t') > F(y(x,t), x', t' )$ for $y\neq y(x,t)$.
\end{lem}
\begin{proof}
The proof follows directly from the proof of Lemma~2.3. in \cite{WHD97}. It also follows from the poof of Lemma~2.4 in \cite{Wang01}, noting that we assumed $\rho_0$ to be strictly positive (at least in the $L_\infty$-sense).
\end{proof}
Now we establish a similar result for the part of the solution depending on the boundary data.
\begin{lem}\label{l2}
For fixed $(x, t)$, $x, t >0$, let $\tau= \tau_{1}$ be a point which minimizes the functional
$G(\tau, x, t)$. Let $(\bar x, \bar t )\neq(x,t)$ be any point on the line segment joining $(x,t)$ and
$(0, t_{1})$. Then the minimizer of $G(\tau,\bar x, \bar t)$ is unique and is $\tau_{1}$.
\end{lem}
\begin{proof}
We want to show that for $\tau\neq\tau_1\colon$
\begin{equation*}
G(\tau, \bar{x}, \bar{t})-G(\tau_1, \bar{x}, \bar{t})>0\,.
\end{equation*}
By definition we have
\begin{multline*}
G(\tau, \bar{x}, \bar{t})-G(\tau_1, \bar{x}, \bar{t})=\int_{\tau_1}^{\tau}
\left[\bar{x}-u_b(\eta)(\bar{t}-\eta)\right]\rho_b(\eta)u_b(\eta)d\eta = \\
=\bar x\int_{\tau_1}^{\tau} \left[1-u_b(\eta)\frac{\bar{t}-\eta}{\bar{x}}\right]\rho_b(\eta)u_b(\eta)d\eta=\\
=\bar x\int_{\tau_1}^{\tau} \left[1-u_b(\eta)\frac{\bar{t}-\tau_1}{\bar{x}} -u_b (\eta) \frac{\tau_1-\eta}{\bar{x}}\right]\rho_b(\eta)u_b(\eta)d\eta\,.
\end{multline*}
Since $(\bar x,\bar t)$ lies on the line connecting $(x,t)$ and $(0,\tau_1)$ we conclude
\begin{multline*}
G(\tau, \bar{x}, \bar{t})-G(\tau_1, \bar{x}, \bar{t})=\bar x \int_{\tau_1}^{\tau} \left[1-u_b(\eta)\frac{t-\tau_1}{x} -u_b (\eta) \frac{\tau_1-\eta}{\bar{x}}\right]\rho_b(\eta)u_b(\eta)d\eta= \notag\\
=\bar x\int_{\tau_1}^{\tau} \left[1-u_b(\eta)\frac{t-\eta}{x}\right]\rho_b(\eta)u_b(\eta)d\eta+ \bar x\int_{\tau_1}^{\tau} u^2_b (\eta)\rho_b(\eta) (\tau_1 -\eta)\left[\frac{1}{x}-\frac{1}{\bar{x}}\right] d \eta\,.
\end{multline*}
Now the first term in the sum is $\tfrac{\bar x}{x}\left[G(\tau, x, t)-G(\tau_1, x, t)\right]$, which is non-negative by assumption. For the second term observe that $\bar x<x$. Thus it is strictly positive if $\tau_1<\tau$ but (considering the direction of integration) also if $\tau_1>\tau$.
\end{proof}
The minima of the initial and boundary potentials, respectively, were introduced in \eqref{e2.6} and \eqref{e2.5} as
\begin{equation*}
F(x,t)=\min_{y\in [0, \infty)}F(y,x,t)\,,
\end{equation*}
\begin{equation*}
G(x,t)=\min_{\tau\in[0,\infty)}G(\tau, x, t)\,.
\end{equation*}
Observe that for a fixed $t>0$ the function $F(x,t)$ is monotonically decreasing in $x$ while $G(x,t)$ is monotonically increasing.\\
\begin{lem}\label{l0}
The function $[0,\infty[\times[0,\infty[\rightarrow \mathbb{R}\colon  (x,t)\mapsto F(x,t)$ is locally Lipschitz continuous and the same holds for $G(x,t)$.
\end{lem}
\begin{proof}
Let $U$ be a bounded open subset of $[0,\infty[\times[0,\infty[$.  Since  $y_* (x, t)$ is  locally bounded in  $[0,\infty[\times[0,\infty[$,  there exists an  $M>0$   such that
\[
\Big|\int_0 ^{y_* (x, t)} \rho_0 (\eta) d\eta \Big| < M
\]
for all $(x,t)\in U$. For $(x_1, t), (x_2, t) \in U$ we have
\begin{multline*}
F(x_1 ,t)-F(x_2, t)= F(y_* (x_1,t), x_1, t)-  F(y_* (x_2,t), x_2, t)=\\
=  [F(y_* (x_1,t), x_1, t)- F(y_* (x_1,t), x_2, t)] +[ F(y_* (x_1,t), x_2, t)-  F(y_* (x_2,t), x_2, t)]\,.
\end{multline*}
Since the second term is non-negative we infer that
\[
F(x_1 ,t)-F(x_2, t)\geq(x_2-x_1) \int_0 ^{y_* (x_1, t)} \rho_0 (\eta) d\eta\,.
\]
Similarly we get
\[
F(x_1 ,t)-F(x_2, t) \leq  (x_2-x_1) \int_0 ^{y_* (x_2, t)} \rho_0 (\eta) d\eta\,.
\]
Combining the inequalities above results in
\[
| F(x_1 ,t)-F(x_2, t)| \leq M |x_1- x_2|\,.
\]
On the other hand varying $t$ we obtain in a similar manner for  $(x, t_1), (x, t_2) \in U$,
\[
| F(x ,t_1)-F(x, t_2)| \leq M |t_1- t_2|\,.
\]
Therefore, for $(x_1, t_1), (x_2, t_2) \in U$ we conclude
\[
| F(x_1 ,t_1)-F(x_2, t_2)| \leq M\big( |x_1- x_2|+ |t_1-t_2|\big)\,.
\]
Lipschitz continuity of $G$ can be checked in a similar manner.
\end{proof}

Next we define the characteristic triangle associated to a point $(x,t)$. We will later show that this triangle contains all the initial or boundary information, respectively, necessary to give the solution at point $(x,t)$.
\begin{definition}\label{d1}
Let $x\geq 0$, $t>0$ and $F(x,t)$, $G(x,t)$ be given by equations \eqref{e2.6}, \eqref{e2.5}.
\begin{enumerate}
\item[1.] For $F(x,t)< G(x,t)$ and $x>0$ we define the characteristic triangle at the point $(x,t)$ as the convex hull generated by the points $(x,t), (y_{*} (x,t), 0)$ and $(y^{*} (x,t), 0)$.\\
\item[2.] For $F(x,t)>G(x,t)$ we define the characteristic triangle at the point $(x,t)$ as the convex hull generated by the points $(x,t), (0, \tau_{*} (x,t),$ and $(0, \tau^{*} (x,t))$.\\
\item[3.] For $F(x,t)=G(x,t)$ we define the characteristic triangle at the point $(x,t)$ as the convex hull generated by the points
 $(x,t)$, $(y^{*} (x,t),0)$, $(0, \tau^{*} (x,t))$ and $(0,0)$.
\item[4.] For $x=0$ and $F(0,t)<G(0,t)$ we define the characteristic triangle as the convex hull generated by the points $(0,t)$, $(0,0)$ and $(y^{*} (x,t), 0)$.
\end{enumerate}
We denote the characteristic triangle associated with the point $(x,t)$ by $\Delta(x,t)$.
\end{definition}
Note that the characteristic triangle may collapse to a line segment or even to a single point (Case 2 with $x=0$). Figure~\ref{bildtriangle} serves as an illustration of possible cases for characteristic triangles.

\begin{figure}[htb]
\centering
\begin{tikzpicture}[scale=0.8]
% axis
\draw[->] (0,0) -- (9,0) node[anchor=north] {$x$};
\draw[->] (0,0) -- (0,6) node[anchor=east] {$t$};
% labels
\draw	(7,-0.05) node[anchor=north] {$\rho_0$, $u_0$};
\node [rotate=90] at (-.3,5) {$\rho_b$, $u_b$};
\draw (0,4) node[anchor=east]{$T$};
\draw[dashed] (0,4) -- (9,4);
\draw (2,0) node[anchor=north] {$ $}
			(1,4.5) node[anchor=south] {$G<F$}
			(3,4.5) node[anchor=south] {$G=F$}
			(5,4.5) node[anchor=south] {$G>F$};
\draw (1.5,4) node[anchor=south] {$x_1$};
\draw	(1.5,4) circle[radius=1.5pt];
\fill (1.5,4) circle[radius=1.5pt];
\draw[pattern=north west lines, pattern color=black] (0,2) to (1.5,4) to (0,3) to (0,2);
\draw (0,2) node[anchor=east] {$\tau_*(x_1,T)$};
\draw	(0,2) circle[radius=1.5pt];
\draw (0,3) node[anchor=east] {$\tau^*(x_1,T)$};
\draw	(0,3) circle[radius=1.5pt];
\draw (3,4) node[anchor=south] {$x_2$};
\draw	(3,4) circle[radius=1.5pt];
\fill (3,4) circle[radius=1.5pt];
\draw[pattern=horizontal lines, pattern color=black] (0,1) to (3,4) to (2,0) to (0,0) to (0,1);
\draw (0,1) node[anchor=east] {$\tau^*(x_2,T)$};
\draw	(0,1) circle[radius=1.5pt];
\node [anchor=center,rotate=90] at (2,-1) {$y^*(x_2,T)$};
\draw	(2,0) circle[radius=1.5pt];
\draw (5,4) node[anchor=south] {$x_3$};
\draw	(5,4) circle[radius=1.5pt];
\fill (5,4) circle[radius=1.5pt];
\draw[pattern=north east lines, pattern color=black] (3,0) to (5,4) to (5,0) to (3,0);
\node [anchor=center,rotate=90] at (3,-1) {$y_*(x_3,T)$};
\draw	(3,0) circle[radius=1.5pt];
\node [anchor=center,rotate=90] at (5,-1) {$y^*(x_3,T)$};
\draw	(5,0) circle[radius=1.5pt];
\end{tikzpicture}
\vspace*{-.3cm}\caption{Illustration of characteristic triangles}\label{bildtriangle}
\end{figure}
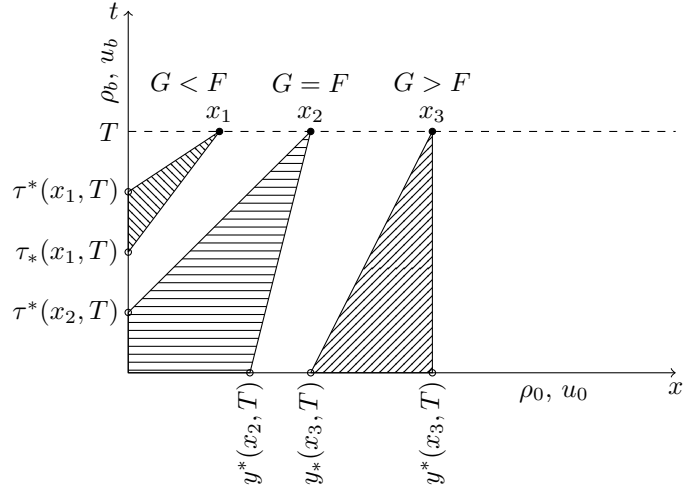
Note that $G=F$ can happen on an interval in $x$ for fixed $t$. Since however $F$ is decaying in $x$ and $G$ is increasing in $x$ this can only happen if both, $F$ and $G$, are constants. We denote this closed interval by
\[
I(t)=\{x|F(x,t)=G(x,t)\}=[l(t),r(t)]\,.
\]
The following Lemma gives a characterization of the set where $F=G$.
\begin{lem}\label{lemF=G}
With the notation as above, let $t$ be such that $\overset{\circ}{I}(t)\neq \emptyset$. Then:
\begin{enumerate}
\item For all $x\in I(t)$ it holds that $F(x,t)=G(x,t)=0$.
\item $\tau_*(x,t)=\tau^*(x,t)=0$ on $]l(t),r(t)]$ and $\tau_*(l(t),t)=0$.
\item $y_*(x,t)=y^*(x,t)=0$ on $[l(t),r(t)[$ and $y_*(r(t),t)=0$.
\item For all $t'<t$ we have $\overset{\circ}{I}(t')\neq 0$.
\item The set $\bigcup_{0\leq t'\leq t} I(t)$ is star-shaped with respect to $(0,0)$.
\end{enumerate}
\end{lem}
\begin{proof}
(1) and (2) are direct consequences of Remark~\ref{rem:const}.\\
For the proof of (3) note that, since $F(x,t)=0$, clearly $y=0$ is a minimizer of $F(y,x,t)$ and thus $y_*(x,t)=0$ on $I(t)$. The statement for $y^*$ then follows from the second point in Lemma~\ref{lnew}.\\
To prove (4) let $x_1<x_2\in\overset{\circ}{I}(t)$ and consider the line segments joining $(x_i,t)$ and $(0,0)$. Now denote the points on these line segments at time $t'<t$ by $(x_1',t')$ and $(x_2',t')$. Then from Lemma~\ref{l1} and Lemma~\ref{l2}
\[
\tau_*(x_1',t')=\tau_*(x_2',t')=y_*(x_1',t')=y_*(x_2',t')=0\,,
\]
and thus $\forall x\in [x_1',x_2']\colon F(x,t')=G(x,t')$, leading to $\overset{\circ}{I}(t')\neq \emptyset$.\\
(5) follows from the proof of (4) immediately.
\end{proof}
\begin{corollary}\label{cor:lr}
If for some $t^\prime>0$ we have that $l(t^\prime)=r(t^\prime)$, then
\[
 \forall t>t^\prime\colon l(t)=r(t)\,.
\]
\end{corollary}
Figure~\ref{bildrarefaction} illustrates the proof of Lemma~\ref{lemF=G} as well as the set G = F, including the characteristic triangles, in the situation of Corollary~\ref{cor:lr}.
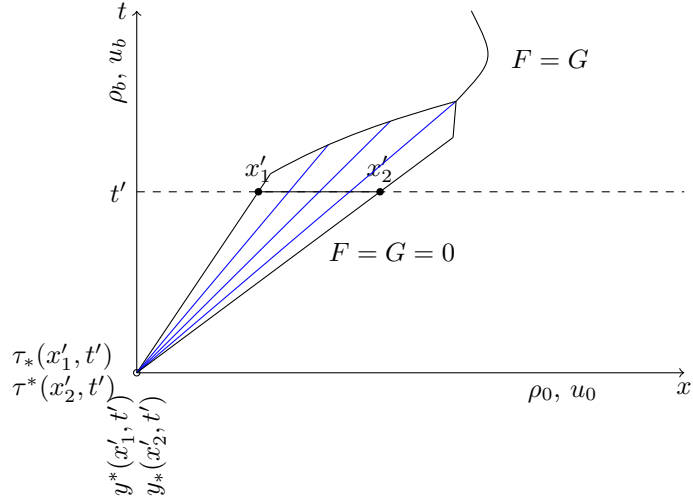
\begin{figure}[htb]
\centering
\begin{tikzpicture}[scale=0.8]
% axis
\draw[->] (0,0) -- (9,0) node[anchor=north] {$x$};
\draw[->] (0,0) -- (0,6) node[anchor=east] {$t$};
% labels
\draw	(7,-0.05) node[anchor=north] {$\rho_0$, $u_0$};
\node [rotate=90] at (-.3,5) {$\rho_b$, $u_b$};
\draw (0,3) node[anchor=east]{$t'$};
\draw[dashed] (0,3) -- (9,3);
\draw (2,3) -- (4,3);

\draw (2,3) node[anchor=south] {$x_1'$};
\draw	(2,3) circle[radius=1.5pt];
\fill (2,3) circle[radius=1.5pt];
\draw (4,3) node[anchor=south] {$x_2'$};
\draw	(4,3) circle[radius=1.5pt];
\fill (4,3) circle[radius=1.5pt];
\draw	(0,0) circle[radius=1.5pt];
\node [anchor=center] at (-0.8,0) {\parbox{2cm}{$\tau_*(x_1',t')$\\$\tau^*(x_2',t')$}};
\node [anchor=center,rotate=90] at (0,-0.8) {\parbox{2cm}{$y^*(x_1',t')$\\$y_*(x_2',t')$}};
\draw (2*1.1,3*1.1) -- (0,0);
\draw [blue] (2.5*1.26,3*1.26) -- (0,0);
\draw [blue] (3*1.39,3*1.39) -- (0,0);
\draw [blue] (3.5*1.5,3*1.5) -- (0,0);
\draw (4*1.3,3*1.3) -- (0,0);
\draw (3.5*1.5,3*1.5) -- (4*1.3,3*1.3);
\draw (2*1.1,3*1.1) .. controls (2.5*1.3,3*1.3) and (3*1.4,3*1.4) .. (3.5*1.5,3*1.5);
\draw (3.5*1.5,3*1.5) .. controls (5.9,5.2 ).. (5.5,6);
\draw (6,5.2) node[anchor=west] {$F=G$};
\draw (3,2) node[anchor=west] {$F=G=0$};
\end{tikzpicture}
\vspace*{-.3cm}
\caption{Characteristic triangles (blue lines) corresponding to a rarefaction wave emanating from the origin}\label{bildrarefaction}
\end{figure}

\noindent Our next goal is to show that, as anticipated in Figure~\ref{bildtriangle}, characteristic triangles associated with different positions at the same time do not intersect. For two triangles reaching only the initial data this is clear from Lemma~\ref{l1}. The same holds true if both triangles only reach the boundary data by Lemma~\ref{l2}. We still need to study the case when one of the triangles is as defined in points (3) or (4) of Definition~\ref{d1}. For this purpose we prove the following lemma.
\begin{lem}\label{lem3}
Let $t>0$ be fixed, and $x_1$, $x_2>0$, $x_1\neq x_2$ but arbitrary. Then the characteristic triangles associated to $(x_1,t)$ and $(x_2,t)$ do not intersect in the interior of $\mathbb{R}_+^2$.
\end{lem}
\begin{proof}
Let $t$ be fixed and assume w.l.o.g. that $x_2>x_1$. Since $F(x,t)$ is decreasing and $G(x,t)$ is increasing in $x$ we only have the following cases:\\
\underline{Case 1}, ${G(x_2,t)<F(x_2,t)}$: Then automatically also $G(x_1,t)<F(x_1,t)$. Now assume that the two characteristic triangles intersect.
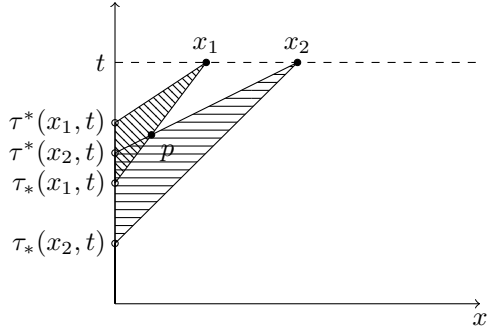
\begin{figure}[htb]
\centering
\begin{tikzpicture}[scale=0.8]
% axis
\draw[->] (0,0) -- (6,0) node[anchor=north] {$x$};
\draw[->] (0,0) -- (0,5) ;
% labels
\draw (0,4) node[anchor=east] {$t$};
\draw[dashed] (0,4) -- (6,4);
\draw (1.5,4) node[anchor=south] {$x_1$};
\draw	(1.5,4) circle[radius=1.5pt];
\fill (1.5,4) circle[radius=1.5pt];
\draw[pattern=north west lines, pattern color=black] (0,2) to (1.5,4) to (0,3) to (0,2);
\draw (0,2) node[anchor=east] {$\tau_*(x_1,t)$};
\draw	(0,2) circle[radius=1.5pt];
\draw (0,3) node[anchor=east] {$\tau^*(x_1,t)$};
\draw	(0,3) circle[radius=1.5pt];
\draw (3,4) node[anchor=south] {$x_2$};
\draw	(3,4) circle[radius=1.5pt];
\fill (3,4) circle[radius=1.5pt];
\draw[pattern=horizontal lines, pattern color=black] (0,2.5) to (3,4) to (0,1) to (0,0) to (0,2.5);
\draw (0,2.5) node[anchor=east] {$\tau^*(x_2,t)$};
\draw	(0,2.5) circle[radius=1.5pt];
\draw (0,1) node[anchor=east] {$\tau_*(x_2,t)$};
\draw	(0,1) circle[radius=1.5pt];
\draw (0.6,2.8) circle[radius=1.5pt];
\fill (0.6,2.8) circle[radius=1.5pt];
\draw (0.6,2.8) node[anchor=north west] {$p$};
\end{tikzpicture}
\caption{Intersecting triangles for boundary potential}\label{bildboundary}
\end{figure}
Since the point of intersection $p=(x_p,t_p)$ lies on the line segment joining $(x_1,t)$ to $(0,\tau_*(x_1,t))$ we know from Lemma~\ref{l2}, that $G(\tau,x_p,t_p)$ attains its minimum for $\tau=\tau_*(x_1,t)$. Since $p$ is also on the line segment joining $(x_2,t)$ to $(0,\tau^*(x_2,t))$ we have that $G(\tau^*(x_2,t),x_p,t_p)=G(\tau_*(x_1,t),x_p,t_p)$. This however contradicts Lemma~\ref{l2}, because the minimizer is not unique.\\
\underline{Case 2}, ${G(x_2,t)=F(x_2,t)}$: If also $G(x_1,t)=F(x_1,t)$ then on an interval the characteristic triangles are lines and do not intersect (see Figure~\ref{bildrarefaction}). Otherwise, if $G(x_1,t)\neq F(x_1,t)$, denote by $y_2$ a minimizer $F(x_2,t)=F(y_2,x_2,t)$ and by $\tau_2$ one with $G(x_2,t)=G(\tau_2,x_2,t)$. Then we have for $y>0$ arbitrary
\[
F(y,x_1,t)>F(y,x_2,t)\geq F(y_2,x_2,t)=G(\tau_2,x_2,t)>G(\tau_2,x_1,t)\,,
\]
and we conclude $G(x_1,t)<F(x_1,t)$. Thus we are in the situation depicted in Figure~\ref{bildinter}.
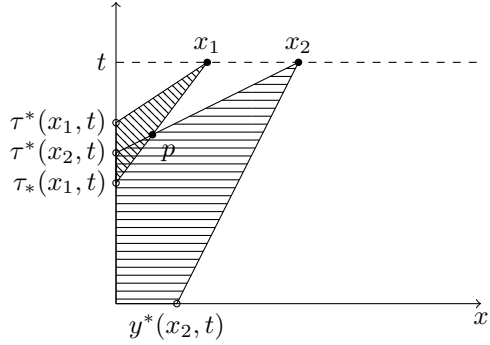
\begin{figure}[h!]
\centering
\begin{tikzpicture}[scale=0.8]
% axis
\draw[->] (0,0) -- (6,0) node[anchor=north] {$x$};
\draw[->] (0,0) -- (0,5) ;
% labels
\draw (0,4) node[anchor=east] {$t$};
\draw[dashed] (0,4) -- (6,4);
\draw (1.5,4) node[anchor=south] {$x_1$};
\draw	(1.5,4) circle[radius=1.5pt];
\fill (1.5,4) circle[radius=1.5pt];
\draw[pattern=north west lines, pattern color=black] (0,2) to (1.5,4) to (0,3) to (0,2);
\draw (0,2) node[anchor=east] {$\tau_*(x_1,t)$};
\draw	(0,2) circle[radius=1.5pt];
\draw (0,3) node[anchor=east] {$\tau^*(x_1,t)$};
\draw	(0,3) circle[radius=1.5pt];
\draw (3,4) node[anchor=south] {$x_2$};
\draw	(3,4) circle[radius=1.5pt];
\fill (3,4) circle[radius=1.5pt];
\draw[pattern=horizontal lines, pattern color=black] (0,2.5) to (3,4) to (1,0) to (0,0) to (0,2.5);
\draw (0,2.5) node[anchor=east] {$\tau^*(x_2,t)$};
\draw	(0,2.5) circle[radius=1.5pt];
\draw (1,0) node[anchor=north] {$y^*(x_2,t)$};
\draw	(1,0) circle[radius=1.5pt];
\draw (0.6,2.8) circle[radius=1.5pt];
\fill (0.6,2.8) circle[radius=1.5pt];
\draw (0.6,2.8) node[anchor=north west] {$p$};
\end{tikzpicture}
\caption{Intersecting triangles}\label{bildinter}
\end{figure}
Again we can conclude as in Case 1 that an intersection is impossible.\\
\underline{Case 3}, ${G(x_2,t)>F(x_2,t)}$:
Here we distinguish three more cases, namely
\begin{enumerate}
\item $G(x_1,t)>F(x_1,t)$. This case is the same as Case 1 but using Lemma~\ref{l1}.
\item $G(x_1,t)=F(x_1,t)$. The argument follows along the lines of Case 2, using Lemma~\ref{l1}.
\item $G(x_1,t)<F(x_1,t)$. Here intersection is not possible by definition.
\end{enumerate}
\end{proof}

\begin{lem}\label{lem4}
Let $t>0$ be fixed, and $x_1>0$. Then the characteristic triangles associated to $(0,t)$ and $(x_1,t)$ do not intersect in the interior of $\mathbb{R}_+^2$.
\end{lem}
\begin{proof}
The proof is trivial if $F(0,t)\geq G(0,t)$. Thus let us assume that  $F(0,t)<G(0,t)$. Then, by monotonicity, we have $F(x,t)<G(x,t)$ for all $x>0$. Thus the characteristic triangle at $(x_1,t)$ is the convex hull of $(x_1,t)$, $(y_*(x_1,t),0)$ and $(y^*(x_1,t),0)$. Since by point (2) of Lemma~\ref{lnew} we have $y^*(0,t)\leq y_*(x_1,t)$ the assertion follows.
\end{proof}
Now we are in the position to state the first theorem that combines the results for fixed $t$ and the lemmas before.
\begin{theorem}\label{l4}
If two characteristic triangles intersect in $\mathbb{R}_+^2$, then one is contained in the other. Moreover, if they intersect on the boundary in more than one point, then also one has to be contained in the other.
\end{theorem}
\begin{proof}
The result follows from the direct application of Lemmas~\ref{l1}--\ref{lem4}.
\end{proof}
The properties of characteristic triangles established so far allow us to derive additional properties of $\tau_*,\tau^*,y_*,y^*$, which will be used frequently later. We collect them in a remark.
\begin{remark}\label{rem:limitstaustar}
(a) At fixed $t$, the function $x\to \tau_*(x,t)$ is right continuous and $x\to \tau^*(x,t)$ is left continuous. Further, $\tau_*(x+,t) = \tau_*(x,t) = \tau^*(x+,t)$ for all $x\geq 0$ and $\tau_*(x-,t) = \tau^*(x,t) = \tau^*(x-,t)$ for $x>0$.

(b) At fixed $t$, the function $x\to y_*(x,t)$ is left continuous and $x\to y^*(x,t)$ is right continuous. Further, $y_*(x+,t) = y^*(x,t) = y^*(x+,t)$ for all $x\geq 0$.

Indeed, to verify (a), combine the semicontinuity properties stated in Lemma~\ref{lnew} with the fact that $\tau(x,t)$ is decreasing in $x$. In particular, $\tau_*(x,t) = \tau_*(x+,t)$. By Lemma~\ref{lem3}, $\tau^*(x+,t) \leq \tau_*(x,t)$. On the other hand, $\tau_*(x+,t) \leq \tau^*(x+,t)$. Combining the inequalities leads to $\tau_*(x+,t) = \tau_*(x,t) = \tau^*(x+,t)$. The second assertion and item (b) is proved in the same way.
\end{remark}
The following lemma states that the domain of interest is indeed covered by characteristic triangles.
\begin{lem}\label{l5}
For any time $t_0 >0$ we have
\[
\bigcup_{x \in [0, \infty)}\Delta(x, t_0)=\{(x,t)| x \in [0, \infty), 0 \leq t \leq t_0\}\,.
\]
\end{lem}
\begin{proof}\underline{Case $I(t_0)\neq\emptyset$:}
Assume first that $I(t_0)$ consists of the single point $x_0 = l(t_0) = r(t_0)$. Let $(x,t)$ be a point which lies left of $\Delta(x_0, t_0)$ (see Figure~\ref{bildl5}). Consider points $(z,t_0)$ on the horizontal line segment joining $(0,t_0)$ with $(x_0,t_0)$. As $z$ decreases to $0$, both $\tau_*(z,t_0)$ and $\tau^*(z,t_0)$ converge to $(0,t_0)$ (Remark \ref{rem:limitstaustar}). Let
\[
    x_1 = \inf\{z: \tau^*(z,t_0) \leq \tau_*(x,t)\}.
\]
Then $\tau^*(x_1+,t_0) \leq \tau_*(x,t)$, and by Remark~\ref{rem:limitstaustar}, $\tau^*(x_1+,t_0) = \tau_*(x_1,t_0)$, so that $\tau_*(x_1,t_0) \leq \tau_*(x,t)$. Whenever $z < x_1$, we have $\tau^*(z,t_0) >\tau_*(x,t)$ and consequently, by the non-intersection property, $\tau^*(z,t_0) \geq \tau^*(x,t)$ as well. Again by Remark~\ref{rem:limitstaustar}, $\tau^*(x_1,t_0) = \tau^*(x_1-,t_0) \geq \tau^*(x,t)$. Thus $\Delta(x,t)\subset\Delta(x_1,t_0)$, as desired.

\begin{figure}[htb]
\centering
\begin{tikzpicture}[scale=1]
% axis
\draw[->] (0,0) -- (6,0) node[anchor=north] {$x$};
\draw[->] (0,0) -- (0,6) node[anchor=east] {$t$};
% labels
\draw	(4.5,-0.05) node[anchor=north] {$\rho_0$, $u_0$};
\node [rotate=90] at (-.3,5) {$\rho_b$, $u_b$};
\draw (0,4) node[anchor=east]{$t_0$};
\draw (0,3.5) node[anchor=east]{$t$};
\draw[dashed] (0,4) -- (6,4);
\draw[dashed] (0,3.5) -- (6,3.5);
\draw[->] (2,4) -- (1.5,4);
\draw (2,4) node[anchor=south] {$z$};
\draw	(2,4) circle[radius=1pt];
\fill (2,4) circle[radius=1pt];
\draw (1,3.5) node[anchor=south] {$x$};
\draw	(1,3.5) circle[radius=1.5pt];
\fill (1,3.5) circle[radius=1.5pt];
\draw[pattern=north west lines, pattern color=black] (0,2) to (1,3.5) to (0,3) to (0,2);

\draw (0,2) node[anchor=east] {$\tau_*(x,t)$};
\draw	(0,2) circle[radius=1.5pt];
\draw (0,3) node[anchor=east] {$\tau^*(x,t)$};
\draw	(0,3) circle[radius=1.5pt];
\draw (3,4) node[anchor=south] {$x_0$};
\draw	(3,4) circle[radius=1.5pt];
\fill (3,4) circle[radius=1.5pt];
\draw[pattern=horizontal lines, pattern color=black] (0,1) to (3,4) to (2,0) to (0,0) to (0,1);
\draw (0,1) node[anchor=east] {$\tau^*(x_0,t_0)$};
\draw	(0,1) circle[radius=1.5pt];
\node [anchor=north] at (2,0) {$y^*(x_0,t_0)$};
\draw	(2,0) circle[radius=1.5pt];
\end{tikzpicture}
\caption{Illustration of Lemma~\ref{l5}}\label{bildl5}
\end{figure}
\noindent
Second, if $l(t_0) \neq r(t_0)$, then $\tau_*(l(t_0),t_0) = 0$ and $y_*(r(t_0),t_0) = 0$ by Lemma~\ref{lemF=G}. We may apply the same arguments to points $(x,t)$ lying to the left of the segment joining $(0,\tau^*(l(t_0),t_0))$ with $(l(t_0),t_0)$ or to the right of the segment joining $(y^*(r(t_0),t_0))$ with $(r(t_0),t_0)$, respectively. If $(x,t)$ lies between those segments, then $x\in I(t)$ and Figure~\ref{bildrarefaction} applies.\\
\underline{Case $I(t_0) =\emptyset$:} Then $\Delta(0,t_0)$ contains the line segment connecting $(0,t_0)$ and $(y^*(0,t_0),0)$ as well as all points to the left of it. For points to the right of the line segment the same argument as in the first case ensures that they are contained in the characteristic triangle of a point $(z,t_0)$.
\end{proof}

\begin{lem}\label{lem:curves}
Let $t_1$ be strictly positive. Each point $(x_1, t_1)$ uniquely determines a curve $x=X(t)$, for $t\geq t_1$, with $x_1= X(t_1)$ such that the characteristic triangles associated to points on the curve form an increasing family of sets.
This curve is Lipschitz continuous as a function of $t\in[t_1,\infty[$. At every $t\geq t_1$, and $(x,t)$ on the curve we have the following:
\begin{enumerate}[label=(\roman*)]
\item \label{it1} If $F(x,t)<G(x,t)$ then
\begin{equation}\label{e2.7}
\lim_{t'', t'\searrow t}\frac{X(t'')-X(t')}{t''-t'}=
\left\{
\begin{aligned}
&\frac{x-y_*(x,t)}{t} &&\text{if}&&y_*(x,t)=y^*(x,t)\,\\
&\displaystyle{\frac{\int_{y_{*}(x,t)}^{y^{*}(x,t)}\rho_0 u_0}{\int_{y_*(x,t)}^{y^*(x,t)}\rho_0} }&&\text{if}&&y_*(x,t)<y^*(x,t)\,.
\end{aligned}
\right.
\end{equation}
\item \label{it2} If $F(x,t)>G(x,t)$ then
\begin{equation}\label{e2.8}
\lim_{t'', t'\searrow t}\frac{X(t'')-X(t')}{t''-t'}=
\left\{
\begin{aligned}
&\frac{x}{t-\tau_*(x,t)} &&\text{if}&&\tau_{*}(x,t)=\tau^{*}(x,t)\,\\
&\displaystyle{\frac{\int_{\tau_*(x,t)}^{\tau^*(x,t)}\rho_b u^2_b}{\int_{\tau_*(x,t)}^{\tau^*(x,t)}u_b\rho_b} }&&\text{if}&&\tau_{*}(x,t)<\tau^{*}(x,t)\,.
\end{aligned}
\right.
\end{equation}
\item \label{it3} If $F(x,t)=G(x,t)$ and $y^*(x,t)\neq0$ or $\tau^*(x,t)\neq 0$ then
\begin{equation}\label{e2.9}
\lim_{t'', t'\searrow t}\frac{X(t'')-X(t')}{t''-t'}=
\frac{\int_{0}^{y^*(x,t)}\rho_0 u_{0}+\int_{0}^{\tau ^*(x,t)}\rho_b u^2_{b}}{\int_{0}^{y^*(x,t)}\rho_0+\int_{0}^{\tau^*(x,t)}u_b\rho_b}\,.
\end{equation}
\item \label{it4} If $F(x,t)=G(x,t)$ and $y^*(x,t)=\tau^*(x,t)=0$, then
\begin{equation*}
\lim_{t'', t'\searrow t}\frac{X(t'')-X(t')}{t''-t'}=\frac{x}{t}
\end{equation*}
\item \label{it5} If $x=0$, $t>0$ and $F(0,t)<G(0,t)$, then
\begin{equation*}
\lim_{t'', t'\searrow t}\frac{X(t'')-X(t')}{t''-t'}=0
\end{equation*}
\end{enumerate}
\end{lem}

\begin{proof}
The statement of \ref{it1} corresponds to Lemma~2.4 in \cite{WHD97} and the proof can be found there.

For the proof of \ref{it2} let $t''>t'>t$ and $X(t'')=x'', X(t')=x'$. The non-intersecting property of characteristic triangles implies the chain of inequalities
\[
\tau_*(x'',t'') \leq \tau_*(x',t') \leq \tau_*(x,t) \leq \tau^*(x,t)  \leq \tau^*(x',t')  \leq \tau^*(x'',t''),
\]
and the semicontinuity then gives that $\tau_*(x'',t'') \to \tau_*(x,t)$ and $\tau^*(x'',t'') \to \tau^*(x,t)$ as $t''\to t$.
Consider the case when $\tau ^{*}(x,t)=\tau_{*}(x,t)$. From Figure~\ref{Bildderivative} it is straightforward to see the following inequality on inclinations:
%%%%%%%%%%%%%%%%%%%%%%%%%%%%%%%%%%%%%%%%%%%%%%%%%%%%%%%%%%%%%%%%%%%%%%%%%%%%%%%%%%%%%%%%%%%%
\begin{figure}[htb]
\centering
\begin{tikzpicture}[scale=0.8]
% axis
\draw[->] (0,0) -- (6,0) node[anchor=north] {$x$};
\draw[->] (0,0) -- (0,5) node[anchor=east] {$t$};
\draw plot [smooth] coordinates {(0,1) (3,3) (3.5,3.3) (4,3.5) (6,4) };
\draw[dashed] (0,1) -- (3,3);
\draw (0,1) node[anchor=east] {$\tau_*(x,t)=\tau^*(x,t)$};
\draw	(0,1) circle[radius=1.5pt];
\fill (0,1) circle[radius=1.5pt];

\draw (3,3) node[anchor=west] {$(x,t)$};
\draw	(3,3) circle[radius=1.5pt];
\fill (3,3) circle[radius=1.5pt];

\draw (3.5,3.3) node[anchor=east] {$(x',t')$};
\draw	(3.5,3.3) circle[radius=1.5pt];
\fill (3.5,3.3) circle[radius=1.5pt];

\draw (4,3.5) node[anchor=south] {$(x'',t'')$};
\draw	(4,3.5) circle[radius=1.5pt];
\fill (4,3.5) circle[radius=1.5pt];

\draw[dashed] (4,3.5) -- (0,3);
\draw (0,3) node[anchor=east] {$\tau^*(x'',t'')$};
\draw	(0,3) circle[radius=1.5pt];
\fill (0,3) circle[radius=1.5pt];

\draw[dashed] (4,3.5) -- (0,0.2);
\draw (0,0.2) node[anchor=east] {$\tau_*(x'',t'')$};
\draw	(0,0.2) circle[radius=1.5pt];
\fill (0,0.2) circle[radius=1.5pt];
\draw (6,4) node[anchor=north] {$X(t)$};
\end{tikzpicture}
\caption{Bounds on slope}\label{Bildderivative}
\end{figure}
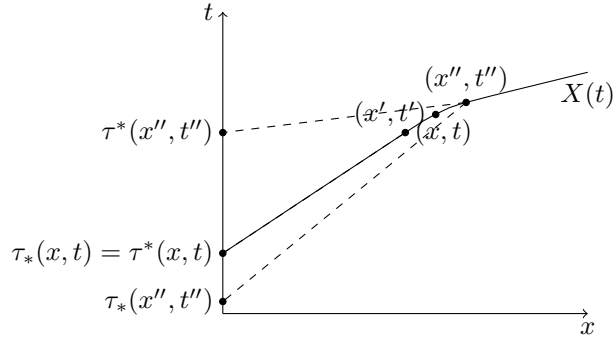
%%%%%%%%%%%%%%%%%%%%%%%%%%%%%%%%%%%%%%%%%%%%%%%%%%%%%%%%%%%%%%%%%%%%%%%%%%%%%%%%%%%%%%%%%%%%
\begin{equation}\label{e2.10}
\frac{x''}{t''-\tau^{*}(x'', t'')} \geq \frac{x''-x'}{t''-t'} \geq \frac{x''}{t''-\tau_{*}(x'', t'')}.
\end{equation}
Passing to the limit as $t'', t' \searrow t$ leads to the first identity of equation \eqref{e2.8}.\\
Now consider $\tau_{*}(x,t)<\tau ^{*}(x,t)$. From the definitions of the boundary functional $G(\tau, y, x, t)$, $\tau_*$ and $\tau^*$, we have
\begin{equation}\label{e2.11}
\begin{split}
G\big(\tau^{*}(x'', t''), x'', t''\big) -&G\big(\tau_{*}(x', t'), x'', t''\big)\leq 0\\
\leq\, &G\big(\tau^{*}(x'', t''), x', t'\big)- G\big(\tau_{*}(x', t'),  x', t'\big).
\end{split}
\end{equation}
After simplification inequality \eqref{e2.11} yields
\begin{equation*}
\int\displaylimits_{\tau_{*}(x', t')} ^ {\tau^{*}(x'', t'')} \!\!\![x''-u_b (\eta)(t''-\eta)] \rho_b (\eta)u_b(\eta) d \eta \leq \int\displaylimits_{\tau_{*}(x', t')} ^ {\tau^{*}(x'', t'')} \!\!\![x'-u_b (\eta)(t'-\eta)] \rho_b (\eta)u_b(\eta) d \eta\,.
\end{equation*}
Thus we can conclude
\begin{equation}\label{e2.13}
\frac{x''- x'}{t''- t'} \leq
\frac{\int_{\tau_{*}(x', t')} ^ {\tau^{*}(x'', t'')}u^2_b (\eta)\rho_b (\eta)d \eta}{\int_{\tau_{*}(x', t')} ^ {\tau^{*}(x'', t'')}u_b(\eta)\rho_b (\eta)d \eta}\,.
\end{equation}
On the other hand considering the inequality
\begin{equation*}
\begin{split}
G\big(\tau_{*}(x'', t''), x'', t''\big) -&G\big(\tau^{*}(x', t'), x'', t''\big) \leq 0 \\
 \leq \,&G\big(\tau_{*}(x'', t''), x', t'\big)- G\big(\tau^{*}(x', t'),x',t'\big),
\end{split}
\end{equation*}
we get, using $\tau^*(x',t') \geq \tau_*(x'',t'')$ that
\begin{equation}\label{e2.14}
\frac{x''- x'}{t''- t'} \geq
\frac{\int_{\tau_{*}(x'', t'')} ^ {\tau^{*}(x', t')}u^2_b (\eta)\rho_b (\eta)d \eta}{\int_{\tau_{*}(x'', t'')} ^ {\tau^{*}(x', t')}u_b(\eta)\rho_b (\eta)d \eta}\,.
\end{equation}
Now passing to the limit as $t'', t'\searrow t$ in equations \eqref{e2.13}  and  \eqref{e2.14}, we proved the second identity of \eqref{e2.8}.

Now to verify the statement of \ref{it3} we assume $F(x,t)=G(x,t)$. Then by definition of the curve and the characteristic triangles, $F(X(t'),t') = G(X(t'),t')$ for all $t' \geq t$. Using the minimizing properties we have the following inequality.
\begin{align*}
&F\big(y^{*}({x'', t''}), x'', t''\big)-G\big(\tau^{*}(x',t') , x'', t''\big)= \\
=\,&G\big(\tau^{*}({x'', t''}), x'', t''\big)-G\big(\tau^{*}(x',t') , x'', t''\big)\leq 0  \\
\leq \,&F\big(y^{*}(x'', t''), x', t'\big)-F\big(y^ {*}(x',t'),  x', t'\big)=\\
=\,&F\big(y^{*}(x'', t''), x', t'\big)-G\big(\tau^ {*}(x',t'),  x', t'\big)\,.
\end{align*}
This inequality implies
\begin{equation*}
\begin{split}
\int_{0}^{y^{*}({x'', t''})}\rho_0(\eta)&u_0(\eta)(t''-t')+\rho_0(\eta)(x'-x'')d\eta  \\
&\leq \int_{0}^{\tau^{*}(x',t')}(x''-x')\rho_b(\eta)u_b(\eta)+ u_b^{2}(\eta)\rho_b(\eta)(t'-t'')d\eta\,,
\end{split}
\end{equation*}
and simplification leads to
\begin{equation}\label{e2.17}
\frac{\int_{0}^{y^{*}({x'', t''})}\rho_0 u_{0}+\int_{0} ^{\tau^{*}(x',t')}\rho_b u^2_{b}}{\int_{0}^{y^{*}(x'', t'')}\rho_0+\int_{0}^{\tau^{*}(x',t')}u_b\rho_b}\leq \frac{x''-x'}{t''-t'}\,.
\end{equation}
Now in the same way starting from the inequality
\begin{equation*}
\begin{split}
G\big(\tau^{*}(x'', t''), x'', t'' \big)-&F\big(y^{*}({x', t'}), x'', t''\big)\leq 0\\
\leq\,& G\big(\tau^ {*}(x'', t''), x', t'\big)- F\big(y^{*}(x', t'), x', t'\big)
\end{split}
\end{equation*}
and simplifying as before we get
\begin{equation}\label{e2.18}
\frac{\int_{0}^{y^{*}({x', t'})}\rho_0 u_{0}+\int_{0} ^{\tau^{*}(x'',t'')}\rho_b u^2_{b}}{\int_{0}^{y^{*}({x', t'})}\rho_0+\int_{0}^{\tau^{*}(x'',t'')}u_b\rho_b}\geq\frac{x''-x'}{t''-t'}\,.
\end{equation}
Passing to the limit as $t'', t'\searrow t$  in \eqref{e2.17} and \eqref{e2.18} completes the proof of \ref{it3}.\\
To prove \ref{it4} observe that in this case the curve $X(t')$ for $t'\leq t$ is just a straight line with inclination $x/t$.\\
In case \ref{it5} the statement is obvious by the definition of characteristic triangles.

Finally, the Lipschitz continuity follows from the fact that, whenever $t'' > t' > t$, the differences $X(t'') - X(t')$ are bounded by a constant times $t'' - t'$, according to \eqref{e2.10}, \eqref{e2.13}, \eqref{e2.14}, \eqref{e2.17} and \eqref{e2.18}.
\end{proof}

\begin{remark}
Note that from \ref{it5} it follows that when such a curve $X(t)$ reaches (with increasing time) $x=0$ it stays there until $F$ and $G$ become equal and then leaves $x=0$ according to \ref{it3}.
\end{remark}
We conclude this sections by \emph{defining} the functions $u(x,t)$ and $m(x,t)$. In the next section we will prove that these are indeed solutions of system \eqref{system_m}.
\begin{definition}\label{d2}
For $x,t>0$ we define the real valued function $u(x,t)$ by
\begin{equation*}
u(x,t)\!=\!\left\{
\begin{aligned}
&\frac{x-y_{*}(x,t)}{t}&&\text{if } F(x,t)<G(x,t) \text{ and } y_{*}(x,t)=y^{*}(x,t)\\
&\dfrac{\int_{y_{*}(x,t)}^{y^{*}(x,t)} \rho_0 u_0 }{\int_{y_{*}(x,t)}^ {y^{*}(x,t)} \rho_0} &&\text{if } F(x,t)< G(x,t)  \text{ and }y_{*}(x,t)<y^{*}(x,t)\\
&\frac{x}{t-\tau_{*}(x,t)}&&\text{if } F(x,t)> G(x,t) \text{ and }\tau_{*}(x,t)=\tau^{*}(x,t)\\
&\dfrac {\int_{\tau_{*}(x,t)}^{\tau^{*}(x,t)} \rho_b u^2_b}{\int_{\tau_{*}(x,t)}^ {\tau^{*}(x,t)} \rho_b u_b} &&\text{if } F(x,t)>G(x,t) \text{ and }\tau_{*}(x,t)<\tau^{*}(x,t)\\
&\frac{\int_{0} ^ {\tau^{*}(x,t)}\rho_b u^2_b+\int_{0} ^ {y^{*}(x,t)}\rho_0 u_0}{\int_{0} ^ {\tau^{*}(x,t)}\rho_b u_b +\int_{0} ^ {y^{*}(x,t)}\rho_0} &&\text{if } F(x,t)=G(x,t)\text{ and  }\!\!\left\{\begin{aligned} &y^*(x,t)\neq0\text{ or}\\ &\tau^*(x,t)\neq0\end{aligned}\right.\\
&\frac{x}{t}&&\text{if } F(x,t)=G(x,t) \text{ and }\!\!\left\{\begin{aligned} &y^*(x,t)=0 \text{ and}\\ &\tau^*(x,t)=0\,.\end{aligned}\right.\\
\end{aligned}
\right.
\end{equation*}
For $x=0$ and $t>0$ we define $u=u_b$ if $F(0,t)> G(0,t)$ and $u=0$ if $F(0,t)<G(0,t)$. For $F(0,t)=G(0,t)$ we use the same definition as for $x>0$.
\end{definition}
More specifically, when $x = 0$ and $F(0,t) = G(0,t)$, the second to last formula applies because $\tau^*(0,t) = t$ always.
\begin{definition}\label{d3}
For $t>0$ and $x\geq0$ we define the real valued function $m(x,t)$ by
\begin{equation*}
m(x,t)=\left\{
\begin{aligned}
&\int_{0}^{y_*(x,t)}\rho_0(\eta)d\eta&&\text{if }  F(x,t)\leq G(x,t)\ \text{ and }x>0\\
-&\int_{0}^{\tau_*(x,t)}\rho_b(\eta)u_b(\eta)d\eta&&\text{if }F(x,t)>G(x,t)\ \text{ or }x=0\,.\\
\end{aligned}\right.
\end{equation*}
\end{definition}
\begin{remark}
Since only (weak) derivatives of $m$ are of interest later, the definition at isolated points is not important apart from $x=0$, where we will use the height of the jump from $m(0,t)$ to $\lim_{x\searrow 0} m(x,t)$ times delta as the mass concentrated at $x=0$. If we have a whole area with $F=G$ then $m$ will be zero, since in this case $F=G=0$ and $y^*=\tau^*=0$ (c.f. Lemma~\ref{lemF=G} ). This is also consistent with tracing back along the triangles (lines) in the rarefaction wave.
\end{remark}
Note that in areas along the $t$-axis $x=0$ where $F\leq G$ we can \emph{not} ask to fulfill the boundary conditions \eqref{e1.3} locally. For all other regions on the boundary, the boundary conditions will be shown to hold, at least under some mild regularity conditions on the boundary data (cf. Section~\ref{sect:bcond}).

\section{Existence of generalized solution}
In this section we are going to show that $(u,m)$ as described in Definitions~\ref{d2} and \ref{d3} satisfy the system of equations
\eqref{system_m}.
For that purpose we first show how to extend the curves defined in Lemma~\ref{lem:curves} to the initial- or boundary manifold.\\
\begin{lem}\label{lem:contcurves}
There is a countable set $S$ of points on the $x$- and $t$-axis with the following properties.
\begin{enumerate}[label=(\roman*)]
\item \label{defcurvesX}For all $(\eta,0)\not\in S$ there is a unique Lipschitz continuous curve $x=X(\eta,t)$, $t\geq0$, such that $X(\eta,0)=\eta$ and the characteristic triangles associated to points on the curve form an increasing family of sets.
\item \label{defcurvesY}For all $(0,\eta)\not\in S$  there is a unique Lipschitz continuous curve $x=Y(\eta,t)$, $t\geq\eta$, such that $Y(\eta,\eta)=0$ and the characteristic triangles associated to points on the curve form an increasing family of sets.
\end{enumerate}
%In the countably many exception points define $X(\eta,t)=X(\eta-,t)$ and $Y(\eta,t)=Y(\eta-,t)$.
Further, for all $\eta>0$ such that $(\eta,0)$ and $(0,\eta)$ does not belong to $S$,
\begin{equation*}
\begin{aligned}
\tfrac{\partial}{\partial t}X(\eta, t)&= u(X(\eta, t), t) &&\text{for almost all $t>0$} \\
\tfrac{\partial}{\partial t}Y(\eta, t)&= u(Y(\eta, t), t) &&\text{for almost all $t>\eta$}\,,
\end{aligned}
\end{equation*}
where the right hand side is a measurable function.
\end{lem}
\begin{proof}
Our proof follows the arguments found in \cite{Wang01} and \cite{WHD97} and extends them to include the boundary points.
We introduce $a^{\pm}(\eta, t)$ and $b^{\pm}(\xi,t)$ to consider all the cases simultaneously. For a fixed point  $(0,\xi)$ on the $t$-axis and $t>0$  we define
\begin{equation*}
\begin{aligned}
&B^-(\xi,t)= \{x\in [0,\infty[\colon F(x,t)\geq G(x,t)\,\, \textnormal{and}\,\, \tau_\ast(x,t) <\xi \}\\
& B^+(\xi,t) =\{x\in [0,\infty[\colon F(x,t)\geq G(x,t)\,\, \textnormal{and}\,\, \tau_\ast(x,t) >\xi\}\,,
\end{aligned}
\end{equation*}
and
\begin{equation*}
\begin{aligned}
b^-(\xi,t)=\begin{cases}
\sup B^-(\xi,t)\,, \,\,& B^-(\xi,t)\neq\emptyset\\
0\,,\,\, &B^-(\xi,t)=\emptyset\,,
\end{cases}
\end{aligned}
\end{equation*}
as well as
\begin{equation*}
\begin{aligned}
b^+(\xi,t)=\begin{cases}
\sup B^+(\xi,t)\,, \,\,& B^+(\xi,t)\neq\emptyset\\
0,\,\, &B^+(\xi,t)=\emptyset\,.
\end{cases}
\end{aligned}
\end{equation*}
Now for a fixed point $(\eta,0)$ on x-axis and $t>0$ we define similarly
\begin{equation*}
\begin{aligned}
&A^-(\eta,t)= \{x\in]0,\infty[\colon F(x,t)\leq G(x,t)\,\, \textnormal{and}\,\, y_\ast(x,t) <\eta \}\\
& A^+(\eta,t) =\{x\in]0,\infty[\colon F(x,t)\leq G(x,t)\,\, \textnormal{and}\,\, y_\ast(x,t) >\eta\}\,,
\end{aligned}
\end{equation*}
and
\[
a^-(\eta,t)=\sup A^-(\eta,t)\,,\quad    a^+(\eta,t)=\inf A^+(\eta,t)\,.
\]
Let us denote
\begin{align*}
&S_{b}(t)=\{(0,\xi)\colon \xi\in (0,\infty), b^-(\xi,t)\neq b^+(\xi,t)\}\\
&S_{a}(t)=\{(\eta,0)\colon \eta\in (0,\infty),  a^-(\eta,t)\neq a^+(\eta,t)\}\,.
\end{align*}

We define the set $S(t)$ by
\[
   S(t) =S_a(t)\cup S_b(t)\,.
\]
Then for any fixed $t>0$ the set $S(t)$ is countable. This follows from the fact that the intervals $[a^-(\eta,t),a^+(\eta,t)]$ and $[a^-(\eta',t),a^+(\eta',t)],$ can not intersect for $\eta\neq\eta'$ except at the endpoints, and the same holds true for $[b^-(\xi,t),b^+(\xi,t)]$ and $[b^-(\xi',t),b^+(\xi',t)]$ for $\xi\neq\xi'$. Indeed, let $\eta'>\eta$ and assume that $a^-(\eta',t) < a^+(\eta,t)$. This means that
\[
\sup\{x\in\mathbb{R}\colon F(x,t)\leq G(x,t)\,\, \textnormal{and}\,\, y_\ast(x,t) < \eta'\} < a^+(\eta,t)\,.
\]
Therefore, for all $x$ between
$a^-(\eta',t)$ and $a^+(\eta,t)$ we have $y_\ast(x,t) \geq \eta' > \eta$, contradicting the definition of $a^+(\eta,t)$. A similar proof works for the other intervals.\\
Observe that for decreasing $t$ the sets $S_a(t)$ form an increasing family of sets. To see this let $0<t' < t$ and $\eta\in S_a(t)$. Then necessarily $y_\ast(x,t) = \eta$ for all $x\in [a^-(\eta,t),a^+(\eta,t)]$. Let $L(x,t)$ be the line connecting $(x,t)$ and $(\eta,0)$ (for such $x$). By Lemma 2.3 of \cite{WHD97} we have that $y_\ast = \eta$ along this line. In particular, the line segment cut out by the bounding lines $L(a^-(\eta,t),t)$ and $L(a^+(\eta,t),t)$ at height $t'$ is contained in $[a^-(\eta,t'),a^+(\eta,t')]$.
Thus $\eta\in S_a(t')$. In particular, for every $t$ there is $n\in\mathbb{N}$ such that $S_a(t)\subset S_a(\frac1n)$. A similar reasoning can be applied to $S_b(t)$.
Therefore, the set
\[
   S = \bigcup_{t>0}S(t)
=\bigcup_{n\in\mathbb{N}} S\big(\tfrac1n\big)
\]
is countable.

Next we discuss the definition of generalized characteristics according to \cite{Wang01}.
For $(\eta,0)\notin S$, the generalized characteristic is defined by
\[
   X(\eta,t) = a^+(\eta,t)\,,\  t>0\,,
\]
with $X(\eta,0) = \eta$. Since $(\eta,0)\not\in S$, $a^-(\eta,t) = a^+(\eta,t)$ for all $t>0$.\\

Let $x=X(\eta,t)$ be the generalized characteristic. We claim that the characteristic triangles at the point $(X(\eta,t), t)$ form an increasing family of sets, that is if $(x_1,t_1)$ and $(x_2,t_2)$ are two points on the curve with $t_1<t_2$, then $\Delta(x_1,t_1)\subset \Delta(x_2,t_2).$
Indeed, in the case when both characteristic triangles are given by $(x_i,t_i)$, $(y_*(x_i,t_i),0)$, $(y^*(x_i,t_i),0)$, consider the line segment
joining the point $(x_2,t_2)$ to $(y^*(x_2,t_2),0)$ and assume that segment intersects $t=t_1$ at the point $(x_3,t_1)$. Then from Lemma~\ref{l1} we have $y^*(x_2,t_2)= y_*(x_3,t_1).$ Then by definition of $X(\eta,t)$, we find
\[
X(\eta,t_1)\leq x_3\,.
\]
On the other hand, assume that the the line segment
joining the points $(x_2,t_2)$ and $(y_*(x_2,t_2), 0)$ intersects the line $t=t_1$ at $(\bar{x}_3,t_1).$ Then we claim
\[
\bar{x}_3\leq X(\eta, t_1)\,.
\]
Indeed, the map $x\to y_*(x,t_2)$ is lower semicontinuous and increasing, hence left continuous. Further, $x_2 = \sup\{x:F(x,t_2)\leq G(x,t_2) {\rm\ and\ } y_*(x,t_2) < \eta\}$. Thus $y_*(x_2,t_2) = \lim_{x\to x_2-}y_*(x,t_2) \leq \eta$, and consequently $\bar{x}_3\leq \eta$ as well.
Combining the two arguments shows that $(x_1,t_1)\in \Delta(x_2,t_2)$ and hence using the non-intersection property of characteristic triangles we have
\[
\Delta(x_1,t_1)\subset \Delta(x_2,t_2)\,.
\]
In the case when either of the characteristic triangles is given by the edges $(x_i,t_i)$, $(0, \tau^*(x_2,t_2))$, $(y^*(x_i,t_i),0)$, the inclusion $\Delta(x_1,t_1)\subset \Delta(x_2,t_2)$ follows directly from the first argument above and the non-intersection property.

Now we turn to the time derivative of the curves $X(\eta,t)$.
It is obvious that $\eta\to X(\eta,t)$ is an increasing function at fixed $t>0$. Therefore, it is Borel (and Lebesgue) measurable.
Lemma~\ref{lem:curves} together with Definition~\ref{d2} mean that
\begin{equation}\label{eq:derivative}
\frac{\partial}{\partial t}X(\eta, t) = u(X(\eta,t),t)
\end{equation}
in the sense of a right derivative, at least for $(\eta,0)\not\in S$. Since $X(\eta,t)$ is Lipschitz continuous, it is differentiable almost everywhere and its derivative satisfies \eqref{eq:derivative}. Further, $u(X(\eta,t),t)$ is a limit of difference quotients of measurable functions, and thus measurable.

Similarly, for $(0,\xi)\notin S$ the generalized characteristic is defined by
\[
Y(\xi,t)=b^+(\xi,t), t>0\,.
\]
An analogous assertion as above is true for the characteristic triangles, and
\[
\frac{\partial}{\partial t}Y(\xi,t)=u(Y(\xi,t),t)
\]
holds for almost all $t> \xi$.
\end{proof}
\begin{remark}
The exceptional set $S$ corresponds to points on the $x$- or $t$-axis from which a rarefaction wave starts. Indeed, if $(\eta,0)\not\in S_a(t)$ and $t>0$, then $y_*(x,t) = \eta = y^*(x,t)$ whenever
$x\in [a_-(\eta,t), a_+(\eta,t)]$. One could define $X(\eta, t)$ by $X(\eta-, t)$ as is done in \cite{Wang97} or by $X(\eta+, t)$ (to include $\eta = 0$). However, to state and prove the results of the present paper, it is not required to assign  a value to $X(\eta,t)$ at the exceptional points.
\end{remark}
\begin{definition}\label{def:qE}
We define, for $x,t>0$ the momentum and the kinetic energy associated to equation~\ref{system_m} by
\begin{equation}\label{e3.1}
q(x,t)=\left\{
 \begin{aligned}
&\int_{0}^{y_*(x,t)}\rho_0(\eta)u_0(\eta)d\eta\,, &\text{if } F(x,t)&\leq G(x,t)\\
-&\int_{0}^{\tau_*(x,t)}\rho_b(\eta)u^2_b (\eta)d\eta\,,&\text{if } F(x,t)&>G(x,t)\,,\\
\end{aligned}\right.
\end{equation}
and
\begin{equation}\label{e3.2}
E(x,t)=\left\{
\begin{aligned}
&\tfrac{1}{2}\int_{0}^{y_*(x,t)}\rho_0(\eta)u_0(\eta)u(X(\eta,t),t)d\eta\,,&\text{if }F(x,t)&\leq G(x,t)\\
-&\tfrac{1}{2}\int_{0}^{\tau_*(x,t)}\rho_b(\eta)u^2_b (\eta)u(Y(\eta,t),t)d\eta\,, &\text{if }F(x,t)&>G(x,t)\,.\\
\end{aligned}\right.
\end{equation}
\end{definition}
The following lemma will make our physical interpretation more precise.
\begin{lem}\label{l51}
In the sense of Radon-Nikodym derivatives in $x$, the following holds in the interior of $\mathbb{R}_+^2$:
(i) $dq=udm$,
(ii) $dE=\frac{1}{2}u^2dm$.
\end{lem}
\begin{proof}
If $F(x,t)< G(x,t)$, then the result is Lemma~2.8. in \cite{Wang01}.\\
Let $(x, t)$ be a point where $G(x,t)<F(x,t)$.
If $\tau_*$ is constant in some neighborhood of $(x,t)$, then the above quantities are constant and the lemma holds trivially. Now suppose $\tau_*(x,t)$ is not constant in a neighborhood of $(x,t)$ and assume
$\tau_*(x,t)= \tau^ *(x,t)=\tau(x,t)$. Let $x_1 < x < x_2$, then by definition
\begin{align*}
G(\tau_{*}(x_1,t), x_1, t)&= \int_{0} ^ {\tau_{*}(x_1,t)}[x_1-(t-\eta)u_b(\eta)]\rho_b(\eta)u_b(\eta)d\eta\\
G(\tau_{*}(x_2,t), x_1 ,t)&= \int_{0} ^ {\tau_{*}(x_2,t)}[x_1-(t-\eta)u_b(\eta)]\rho_b(\eta)u_b(\eta)d\eta\,.
\end{align*}
By the minimizing properties we have we have
\[
G(\tau_{*}(x_1,t), x_1, t) \leq G(\tau_{*}(x_2,t), x_1 ,t)\,,
\]
and using the definitions and the fact that $\tau_{*}(x_2,t) < \tau_{*}(x_1,t)$ leads to
\begin{equation}\label{e3.5}
\frac{x_1}{t-\tau_{*}(x_2,t)} \leq \frac{\int_{\tau_{*}(x_2,t)}^{\tau_{*}(x_1,t)}\big(\frac{t-\eta}{t-\tau_{*}(x_2,t)}\big)u^2_b(\eta)\rho_b(\eta)d\eta}{\int_{\tau_{*}(x_2,t)}^{\tau_{*}(x_1,t)}\rho_b(\eta)u_b(\eta)d\eta}\,.
\end{equation}
Now since, in the upper integral,
\[
t - \eta \leq t - \tau_{*}(x_2,t),
\]
and since $\tau_*(x,t) = \tau^*(x,t) = \tau(x,t)$ is continuous at $(x,t)$ we can take the limits $x_1\nearrow x$ and $x_2\searrow x$ in \eqref{e3.5} to derive
\begin{equation}\label{e3.6}
\frac{x}{t-\tau(x,t)}\leq \lim_{x_2, x_1 \to x}\frac{q(x_1,t)-q(x_2,t)}{m(x_1,t)-m(x_2,t)}\,.
\end{equation}
Similarly, considering the inequality
\begin{equation*}
G(\tau_{*}(x_2,t), x_2, t) \leq G(\tau_{*}(x_1,t), x_2 ,t)
\end{equation*}
and following the analysis as above we get
\begin{equation}\label{e3.7}
\frac{x}{t-\tau(x,t)}\geq \lim_{x_2, x_1 \to x}\frac{q(x_1,t)-q(x_2,t)}{m(x_1,t)-m(x_2,t)}\,.
\end{equation}
From equations \eqref{e3.6} and \eqref{e3.7} and Definition~\eqref{d2}, we conclude $dq= u dm$.\\
If  $\tau_{*}(x,t)< \tau^{*}(x,t)$, then
\[
\lim_{x_2, x_1 \to x}\frac{q(x_2,t)-q(x_1,t)}{m(x_2,t)-m(x_1,t)}
=\lim_{x_2, x_1 \to x}\frac{\int_{\tau_{*}(x_2,t)} ^ {\tau_{*}(x_1,t)}\rho_b u_b^2}{\int_{\tau_{*}(x_2,t)} ^ {\tau_{*}(x_1,t)}\rho_b u_b}
= \frac{\int_{\tau_{*}(x,t)} ^ {\tau^{*}(x,t)}\rho_b u^2_b}{\int_{\tau_{*}(x,t)} ^ {\tau^{*}(x, t)}\rho_bu_b}\,.
\]
Here we used Remark~\ref{rem:limitstaustar}.
Now we consider the remaining case, where $(x, t)$ is a point with $F(x,t)= G(x,t)$. If this happens in an isolated point then we have, for $x_1<x<x_2$
\begin{multline*}
\lim_{x_2, x_1 \to x}\frac{q(x_2,t)-q(x_1,t)}{m(x_2,t)-m(x_1,t)}=\\
=\lim_{x_2, x_1 \to x}\frac{\int_{0} ^ {\tau_{*}(x_1,t)}\rho_b u^2_b+\int_{0} ^ {y_{*}(x_2,t)}\rho_0 u_0}{\int_{0} ^ {\tau_{*}(x_1,t)}\rho_bu_b +\int_{0} ^ {y_{*}(x_2,t)}\rho_0}
= \frac{\int_{0} ^ {\tau^{*}(x,t)}\rho_b u^2_b+\int_{0} ^ {y^{*}(x,t)}\rho_0 u_0}{\int_{0} ^ {\tau^{*}(x,t)}\rho_bu_b +\int_{0} ^ {y^{*}(x,t)}\rho_0}\,,
\end{multline*}
again using Remark~\ref{rem:limitstaustar}.
If on the other hand $F(x,t)=G(x,t)$ in a whole neighborhood of $(x,t)$ then we are in a rarefaction wave emanating from zero and thus $\tau^*(x,t)=y^*(x,t)=0$ in a whole neighborhood. Then $m$, $q$ and $E$ are zero by definition and the proof is finished. In the boundary points of the region $F(x,t)=G(x,t)$, the same proof as in the case $F<G$ or $F>G$ works, on the right and left boundary, respectively.
Thus in all possible cases we derived that $dq= u dm$ in the sense of Radon-Nikodym derivative.

Now we turn our attention to the proof of $dE=\frac{1}{2} u^2 dm$.
First let $(x,t)$ again be a point where $G(x,t) < F(x,t)$ and $\tau_{*}(x,t)=\tau^{*}(x,t)$. Then
\begin{equation}\label{e3.10}
E(x_2, t)-E(x_1, t)= \tfrac{1}{2}\int_{\tau_{*}(x_2, t)} ^ {\tau_{*}(x_1, t)} \rho_b (\eta) u^2_b (\eta) u(Y(\eta, t), t)d\eta\,.
\end{equation}
Note that for  $\tau_{*}(x_2, t) \leq \eta \leq \tau_{*}(x_1, t)$ we have
\begin{equation}\label{e3.11}
\frac{x}{t- \tau_{*}(x_2, t)} \leq u(Y(\eta, t), t) \leq  \frac{x}{t- \tau_{*}(x_1, t)}\,.
\end{equation}
Hence from equation \eqref{e3.10} and \eqref{e3.11} we derive
\[
\frac{1}{2}\frac{x}{t- \tau_{*}(x_2, t)} \leq \frac{E(x_2, t)-E(x_1, t)}{q(x_1, t)-q(x_2, t)} \leq \frac{1}{2} \frac{x}{t- \tau_{*}(x_1, t)}\,.
\]
In the limit $x_1\nearrow x\swarrow x_2$, we have $\frac{E(x_2, t)-E(x_1, t)}{q(x_1, t)-q(x_2, t)} \to \frac{1}{2}u$ and
we know $\frac{dq}{dm}= u$. Combining these two, we get $dE= \frac{1}{2} u^2 dm$.\\
For the final case assume $F(x,t)=G(x,t)$ in an isolated point. Then noting that $(x_1,t)$ lies left of $(x,t)$ and thus $F(x_1,t)>G(x_1,t)$, and the opposite holds true for $(x_2,t)$, we have
\begin{equation*}
\begin{split}
&\lim_{x_1, x_2 \to x}\frac{E(x_2, t)-E(x_1, t)}{q(x_2, t)-q(x_1, t)}=\\
&\ \ \,= \frac{\frac{1}{2}\int_{0}^{y^* (x,t)}\rho_0(\eta)u_0(\eta)u(X(\eta,t),t)d\eta+\frac{1}{2}\int_{0}^{\tau^{*} (x,t)}\rho_b(\eta)u^2_b (\eta)u(Y(\eta,t),t)d\eta}
{\frac{1}{2}\int_{0}^{y^* (x,t)}\rho_0(\eta)u_0(\eta)d\eta+\frac{1}{2}\int_{0}^{\tau^{*} (x,t)}\rho_b(\eta)u^2_b (\eta) d\eta}
\end{split}
\end{equation*}
For $\eta \in [0, y^ * (x,t))]$, we have $u(X(\eta,t),t)=u(x,t)$ and similarly for  $\eta \in [0, \tau^ * (x,t))]$ we know that $u(Y(\eta,t),t)=u(x,t)$.\\
Thus from the equation above we have
\[
\lim_{x_1, x_2 \to x}\frac{E(x_2, t)-E(x_1, t)}{q(x_2, t)-q(x_1, t)}=\frac{1}{2}u(x,t)\,.
\]
Since $dq =u dm$,  we again conclude  $dE=\frac{1}{2} u^2$.
This completes the proof.
\end{proof}

\begin{lem} \label{lem:integrals}
Define
\[
\mu(x,t)= \min (F(x,t), G(x,t))\,,
\]
then the following  holds for $x_1,x_2,t>0\colon$
\begin{equation}\label{e3.13}
\int_{x_1}^{x_2} m(x,t)dx= \mu(x_1,t)-\mu(x_2, t)\,.
\end{equation}
\begin{equation}\label{e3.14}
\int_{t_1}^{t_2} q(x,t) dt =\mu(x, t_2)- \mu (x, t_1)\,.
\end{equation}
\end{lem}

\begin{proof}
We start with proving the first equality. For that purpose let $t>0$ be fixed and pick any two points $x, x^{\prime} \in [x_1, x_2], \, x< x^{\prime}$. We claim that
\begin{equation}\label{e3.21}
(x-x^{\prime}) m (x^{\prime}, t) \leq \mu(x^{\prime},t)-\mu(x, t)\leq (x-x^{\prime}) m (x, t)\,.
\end{equation}
For $\mu(x, t)$ and $\mu(x^{\prime}, t)$ depending upon the minimization the possible cases are:
\begin{equation}\label{case}
\begin{aligned}
&(a)\, &&\mu(x,t)=F(x,t)\,,\ &\mu(x^{\prime},t)&=F(x^{\prime},t)\\
&(b)\, &&\mu(x,t)=G(x,t)\,,\ &\mu(x^{\prime},t)&=F(x^{\prime},t)\\
&(c)\, &&\mu(x,t)=G(x,t)\,,\ &\mu(x^{\prime},t)&=G(x^{\prime},t)
\end{aligned}
\end{equation}
The proof of the inequality \eqref{e3.21} for the  case $(a)$ in \eqref{case} can be found in  Lemma~2.9. in \cite{Wang01}. In case (b) we have
\begin{multline*}
\mu(x^{\prime},t)-\mu(x,t)=F(y_*(x^{\prime},t), x^{\prime}, t)-G(\tau_*(x,t), x, t)=\\
=[F(y_*(x^{\prime},t), x^{\prime}, t)-F(y_*(x^{\prime},t), x, t)]+[F(y_*(x^{\prime},t), x, t)-G(\tau_*(x,t), x, t)]\,.
\end{multline*}
Since the term in the second bracket is positive, we get
\begin{equation}\label{eq3.17}
\mu(x^{\prime},t)-\mu(x,t)\geq F(y_*(x^{\prime},t), x^{\prime}, t)-F(y_*(x^{\prime},t), x, t).
\end{equation}
One can also write
\begin{multline*}
\mu(x^{\prime},t)-\mu(x,t)=F(y_*(x^{\prime},t), x^{\prime}, t)-G(\tau_*(x,t), x, t)=\\
=[F(y_*(x^{\prime},t), x^{\prime}, t)-G(\tau_*(x,t), x^{\prime}, t)]+[G(\tau_*(x,t), x^{\prime}, t)-G(\tau_*(x,t), x, t)]\,.
\end{multline*}
Since the term in first bracket is negative, we get
\begin{equation}\label{eq3.18}
\mu(x^{\prime},t)-\mu(x,t)\leq G(\tau_*(x,t), x^{\prime}, t)-G(\tau_*(x,t), x, t)\,.
\end{equation}
Combining \eqref{eq3.17}-\eqref{eq3.18} and using the definition of $m$, $F$ and $G$, we conclude \eqref{e3.21}.\\

In case(c) we have
\begin{multline*}
\mu(x^{\prime},t)-\mu(x, t)=G(\tau_{*}(x^{\prime},t), x^{\prime}, t)-G(\tau_{*}(x,t), x, t)\\
=[G(\tau_{*}(x^{\prime},t), x^{\prime}, t)-G(\tau_{*}(x,t), x^{\prime}, t)]+[G(\tau_{*}(x,t), x^{\prime}, t)-G(\tau_{*}(x,t), x, t))]\,.
\end{multline*}
Since the first bracket of the above expression is negative, we get
\begin{equation} \label{e3.17}
 \mu(x^{\prime},t)-\mu(x, t) \leq G(\tau_{*}(x,t), x^{\prime}, t)-G(\tau_{*}(x,t), x, t)\,.
\end{equation}
On the other hand we write
\begin{multline*}
\mu(x^{\prime},t)-\mu(x, t)=G(\tau_{*}(x^{\prime},t), x^{\prime}, t)-G(\tau_{*}(x,t),  x, t)\\
= [G(\tau_{*}(x^{\prime},t), x^{\prime}, t)-G(\tau_{*}(x^{\prime},t), x , t)]+[G(\tau_{*}(x^{\prime},t), x, t)-G(\tau_{*}(x,t), x, t))]\,.
\end{multline*}
Now the second bracket of the expression is positive and thus
\begin{equation}\label{e3.19}
\mu(x^{\prime},t)-\mu(x, t) \geq G(\tau_{*}(x',t), x^{\prime}, t)-G(\tau_{*}(x',t), x, t)\,.
\end{equation}
Again, combining \eqref{e3.17}, \eqref{e3.19} and using the definitions of $G$ and $m$ we have \eqref{e3.21}.

Since $m$ is monotonous in $x$, it is also Riemann integrable. Taking Riemann sums and using \eqref{e3.21}, we get
\[
\int_{x_1}^{x_2} m(x,t)dx= \mu(x_1,t)-\mu(x_2, t)\,,
\]
finishing the proof of \eqref{e3.13}.\\

\noindent For the proof of \eqref{e3.14} let $x>0$ be fixed and pick $t< t^{\prime}$ in $[t_1, t_2]$. We claim that
\begin{equation}\label{eq3.20}
(t'-t)q(x,t')\leq\mu(x,t')-\mu(x,t)\leq(t'-t)q(x,t)\,.
\end{equation}
To verify this,  we distinguish the following possibilities for $\mu(x,t)$ and $\mu(x,t^{\prime})$:
\[
\begin{aligned}
&(a)\,&&\mu(x,t)=F(x,t)\,, \ &\mu(x,t^{\prime})&=F(x,t^{\prime})\\
&(b)\,&&\mu(x,t)=G(x,t)\,,\ &\mu(x,t^{\prime})&=G(x,t^{\prime})\\
&(c)\,&& \mu(x,t)=F(x,t)\,,\ &\mu(x,t^{\prime})&=G(x,t^{\prime})\\
&(d)\,&& \mu(x,t)=G(x,t)\,,\ &\mu(x,t^{\prime})&=F(x,t^{\prime})
\end{aligned}
\]
Note that again the case $(a)$ on $\{x\}\times[t_1,t_2]$ is covered by Lemma~2.9. in \cite{Wang01}.
Since the proofs in all cases are rather similar, we only present the proof of case $(b)$ explicitly. We start by observing that
\begin{multline*}
\mu(x,t^{\prime})-\mu(x, t)=G(\tau_{*}(x,t^{\prime}), x, t^{\prime})-G(\tau_{*}(x,t), x, t)=\\
= [G(\tau_{*}(x,t^{\prime}), x, t^{\prime})-G(\tau_{*}(x,t), x, t^{\prime})]+[G(\tau_{*}(x,t), x, t^{\prime})-G(\tau_{*}(x,t), x, t))]\leq\\
\leq G(\tau_{*}(x,t), x, t^{\prime})-G(\tau_{*}(x,t), x, t))\,,
\end{multline*}
and
\begin{multline*}
\mu(x,t^{\prime})-\mu(x, t)=G(\tau_{*}(x,t^{\prime}), x, t^{\prime})-G(\tau_{*}(x,t), x, t)=\\
=[G(\tau_{*}(x,t^{\prime}), x, t^{\prime})-G(\tau_{*}(x,t^{\prime}), x, t)]+[G(\tau_{*}(x,t^{\prime}), x, t)-G(\tau_{*}(x,t), x, t))]\geq\\
\geq G(\tau_{*}(x,t^{\prime}), x, t^{\prime})-G(\tau_{*}(x,t^{\prime}), x, t)\,.
\end{multline*}
Those two inequalities combined imply \eqref{eq3.20}.

For a fixed $x$ the function $y_*(x,t)$ is monotone in the interval $[t_1,t_2]$ and thus $q(x,t)$ is a function of bounded variation, hence Riemann integrable.
Now following a similar argument as before, identity \eqref{e3.14} follows from \eqref{eq3.20}.
\end{proof}

From the previous lemma we have $\mu_x= -m$ and $\mu_t = q$, and thus we verified the first equation of system \eqref{System_mqE}. As anticipated in the introduction, for a test function $\varphi$ with compact support in $]0,\infty[^2$ we infer using Lemma~\ref{l51}:
\begin{multline}\label{e3.22}
0=\iint [-\varphi_t\mu_x (x, t) + \varphi_x\mu_t(x, t)] dx dt=\\
=\iint [\varphi_t(x,t) m(x,t) + \varphi_x(x,t) q(x,t)] dx dt=\\
=\iint \varphi_t(x,t) m(x,t) dx dt - \iint \varphi(x,t)u(x,t) m(dx,t) dt\,.
\end{multline}
This identity proves that $(u, m)$ satisfies the first equation of the system \eqref{system_m}.\\
To prove the second equation, we use the following notation.
\begin{lem}\label{lem:qEsol}
For $x, t>0$,  let us denote
 \begin{equation*}
 H(x, t)=\left\{
 \begin{aligned}
 &H_1(x,t)=\int_0 ^ {y_{*} (x,t)} \rho_0 (\eta) u_0 (\eta) (X(\eta, t)-x) d\eta\,, &&\text{if }  F(x,t)\leq G(x,t)\\
&H_2(x,t)=-\int_0 ^ {\tau_{*}(x,t)} \rho_b (\eta) u^2_b (\eta) (Y(\eta, t)-x) d\eta\,, &&\text{if } F(x,t)>G(x,t)\,.
 \end{aligned}\right.
 \end{equation*}
Then we have $H_x = - q $ and $H_t = 2 E$ in the weak sense.
\end{lem}
\begin{proof}
We will first show that
\begin{equation}\label{e3.24}
-\int_{x_1}^{x_2} q(x,t)dx= H(x_2,t)-H(x_1, t)\,.
\end{equation}
Let $t$ be fixed and $[x_1, x_2]$ be an interval. Assume that $x$ and $x^{\prime}$ are any two points in $[x_1,x_2]$ with $x<x^\prime.$ We will again argue by taking Riemann sums. First depending on the minimization we distinguish the cases
\[
\begin{aligned}
&(a)\,&&H(x_1,t)=H_1(x_1,t)\,,\ &&H(x_2,t)=H_1(x_2,t)\\
&(b)\,&&H(x_1,t)=H_2(x_1,t)\,,\ &&H(x_2,t)=H_2(x_2,t)\\
&(c)\,&&H(x_1,t)=H_2(x_1,t)\,,\ &&H(x_2,t)=H_1(x_2,t)\,.
\end{aligned}
\]
For case $(a)$ it is shown in \cite{WHD97} (what we denote by $H_1$ is denoted by $\theta$ in \cite{WHD97}), that \eqref{e3.24} holds.\\
For studying case $(b)$, we define $H_2(\tau,x,t)=-\int_0 ^ {\tau} \rho_b (\eta) u^2_b (\eta) (Y(\eta, t)-x) d\eta$. Since $Y(\eta, t)$ is decreasing in $\eta$ (by the non-intersecting property of the characteristic triangles) and $Y(\eta, t)=x$ for $\tau_{*}(x, t)< \eta< \tau^{*}(x, t)$, we have
\begin{equation}\label{e3.26}
H_2 (x,t)= \min_{\tau\geq 0} H_2 (\tau,x,t)\,.
\end{equation}
In particular, $H_2(\cdot,t)$ is upper semicontinuous. Note that $F(x,t) > G(x,t)$ and hence $H(x,t) = H_2(x,t)$ for all $x$ in the interval $[x_1,x_2]$.
The rest of the proof is similar to the one of Lemma~\ref{lem:integrals}.
We have, denoting minimizers in the usual way,
\begin{multline*}
H(x, t)-H(x^{\prime}, t)= H_2(x, t)-H_2(x^{\prime}, t)
       = H_2 (\tau_{*}(x,t), x, t)- H_2(\tau_{*}(x^{\prime},t),x^{\prime}, t)=\\
       =[H_2 (\tau_{*}(x,t), x, t)-H_2 (\tau_{*}(x^{\prime},t), x, t)]+ [H_2 (\tau_{*}(x^{\prime},t), x, t)-H_2(\tau_{*}(x^{\prime},t),x^{\prime}, t)]\leq\\
			\leq H_2 (\tau_{*}(x^{\prime},t), x, t)-H_2(\tau_{*}(x^{\prime},t),x^{\prime}, t) =-(x-x^{\prime})q(x^{\prime},t)\,.
\end{multline*}
On the other hand
\begin{multline*}
H(x, t)-H(x^{\prime}, t)= H_2(x, t)-H_2(x^{\prime}, t)
       = H_2 (\tau_{*}(x,t), x, t)- H_2(\tau_{*}(x^{\prime},t),x^{\prime}, t)=\\
       =[H_2 (\tau_{*}(x,t), x, t)-H_2 (\tau_{*}(x,t), x^{\prime}, t)]+ [H_2 (\tau_{*}(x,t), x^{\prime}, t)-H_2(\tau_{*}(x^{\prime},t),x^{\prime}, t)]\geq\\
			\geq H_2 (\tau_{*}(x,t), x, t)-H_2 (\tau_{*}(x,t), x^{\prime}, t)=-(x-x^{\prime})q(x,t)\,.
\end{multline*}
Combining those two inequalities establishes
\begin{equation}\label{e3.29.2}
 -(x- x^{\prime}) q(x, t)\leq H(x, t)-H(x^{\prime}, t) \leq -(x- x^{\prime}) q(x^{\prime}, t)\,.
\end{equation}
Now taking the supremum of the Riemann sums over all partitions of the interval $[x_1, x_2]$, we deduce \eqref{e3.24} on all intervals $[x_1,x_2]$ with $F>G$.

The case $F(x,t) = G(x,t)$ requires special consideration. Recall from Lemma~\ref{lemF=G} that this happens on an interval $I(t) = [l(t),r(t)]$. If $l(t) = r(t) = x$, then $\Delta(x,t)$ contains $[0, y^ * (x,t)]\times\{0\}$ and $\{0\}\times [0, \tau^ * (x,t)]$. Thus by the non-intersecting property, we have $X(\eta,t)=x$ for $\eta \in [0, y^ * (x,t)]$ and $Y(\eta,t) =x$ for $\eta \in [0, \tau^ * (x,t))]$. This implies $H_1 (x,t)=0=H_2(x,t)$.
Note that this is also true if $(x,t)$ lies in a rarefaction wave emanating from $0$. In this case we have $y^ * (x,t)=\tau^ * (x,t)=0$ for all points in the interior of $I(t)$ and thus $H$ as well as $q$ is zero in these points. Further,
$\Delta(l(t),t)$ contains $\{0\}\times [0, \tau^ * (l(t),t)]$ and $\Delta(r(t),t)$ contains $[0, y^ * (r(t),t)]\times\{0\}$, so $H$ is also zero in the boundary points of $I(t)$.\\
The proof in case $(c)$ then follows from these facts.
Since $F(x_1,t) > G(x_1,t)$ and $F(x_2,t) \leq G(x_2,t)$, we have $I(t)\subset ]x_1,x_2]$. Write $x_3 = l(t)$, $x_4 = r(t)$. We split the integral in three (possibly only two) parts and use the results of cases $(a)$ and $(b)$ accordingly,
\begin{multline*}
-\int_{x_1}^{x_2}q(x,t)dx=-\int_{[x_1,x_3[}q(x,t)dx -\int_{[x_3,x_4]}q(x,t)dx -\int_{]x_4,x_2]}q(x,t)dx=\\
=H_2(x_3-,t)-H_2(x_1,t)+ 0 + H_1(x_2,t)-H_1(x_4+,t)=H(x_2,t)-H(x_1,t)\,.
\end{multline*}
Here we used the assertions of Remark~\ref{rem:limitstaustar}, namely $\tau_*(x_3-,t) = \tau^*(x_3,t)$ and $y_*(x_4+,t) = y^*(x_4,t)$. The consideration above then allows us to conclude that $H_2(x_3-,t) = H_1(x_4+,t) = 0$.

Collecting all cases proves that $H_x= -q$ weakly for fixed $t$.
Similar arguments can be used to show $H_t= 2 E$.
\end{proof}
Now we conclude that, again for a test function with compact support in $]0,\infty[^2$
\begin{multline}\label{e3.30}
0=\iint H(x,t)(-\psi_{xtx}+\psi_{txx}) dx dt =\iint [H_x \psi_{tx} (x, t)- H_t \psi_{x x}(x, t)] dx dt=\\
=\iint [-q(x,t) \psi_{t x}(x, t)- 2E \psi_{x x}(x, t)] dx dt=\\
=\iint u (x,t) \psi_t (x, t)dm dt +\iint u^2(x,t)\psi_x(x, t)dm dt\,.
\end{multline}
Identity \eqref{e3.30} proves that $(u,m)$ satisfies the second equation of the system \eqref{system_m}.
Combining \eqref{e3.22} and \eqref{e3.30} we proved the following theorem:
\begin{theorem}
The functions $u$ and $m$ as given in Definiton~\ref{d2} and \ref{d3} respectively, are global solutions of \eqref{system_m} in the sense specified in the introduction.
The functions $m$, $q$ and $E$ given in Definition~\ref{d3} and Definition~\ref{def:qE} are global weak solutions of system \eqref{System_mqE} on $\mathbb{R}_+^2$.
\end{theorem}
\begin{proof}
The statement about solutions of \eqref{System_mqE} is clear from Lemma~\ref{lem:integrals} and \ref{lem:qEsol}. Now by Lemma~\ref{l51} these functions and $u$ are related in the correct way as Radon-Nikodym derivatives. Thus following the discussion in the introduction $u$ and $m$ are solutions of \eqref{system_m}.
\end{proof}
Note that we did not discuss initial and boundary data yet. In the next section we will show that $\rho$, as the derivative of $m$, and $u$ satisfy the boundary and initial conditions \eqref{e1.2}, \eqref{e1.3} in an appropriate sense.

\section{Verification of initial and boundary condition} \label{sect:bcond}
Now we turn our attention to the initial and boundary conditions. For that purpose we define, as already discussed in the introduction, the Radon measure $\rho$ as the derivative of $m$. We also explicitly define the mass at $x=0$ as the one sided distributional derivative of $m$ and will show later in this chapter that this leads to conservation of mass.
\begin{definition}\label{def:sol}
Let $m$ be as in Definition~\ref{d3}. Then we define
\begin{equation}
\rho(x,t)=\left\{\begin{aligned}
&\partial_x m&&\text{for } x,t>0\\
&\lim_{x\searrow 0}\rho(x,t) &&\text{for }x=0\text{ and }F(0,t)>G(0,t)\\
&\delta\cdot\lim_{x\searrow 0}\left(m(x,t)-m(0,t)\right)&&\text{for }x=0\text{ and }F(0,t)\leq G(0,t)\,.
\end{aligned}\right.
\end{equation}
Here $\delta$ is the Dirac measure, $\partial_x$ is the distributional derivative and $\rho$ is interpreted as a measure.
\end{definition}
\begin{lem} \label{lem:vel0}
For $u$ according to Definition~\ref{d2} and for all $t>0$ we have
\begin{enumerate}
\item $F(0,t)<G(0,t)\Longrightarrow u(0+,t)<0$ \label{FG1}
\item $u(0+,t)<0\Longrightarrow F(0,t)\leq G(0,t)\,.$ \label{FG2}
\end{enumerate}
\end{lem}
\begin{proof}
We start by proving the first implication. Since $F$ and $G$ are continuous in $(x,t)$, the statement $F(x,t)<G(x,t)$ holds true in a whole neighborhood of the point $(0,t)$. Now following Definition~\ref{d2} in this neighborhood $u$ is defined as
\[
u(x,t)=\left\{\begin{aligned}
&\frac{x-y_*(x,t)}{t}&&\text{if } y_{*}(x,t)=y^{*}(x,t)\\
&\dfrac{\int_{y_{*}(x,t)}^{y^{*}(x,t)} \rho_0 u_0 }{\int_{y_{*}(x,t)}^ {y^{*}(x,t)} \rho_0} &&\text{if } y_{*}(x,t)<y^{*}(x,t)\,.\\
\end{aligned}\right.
\]
In the first case we see that for $x\searrow 0$ the velocity $u$ becomes negative, apart from the case when $\lim_{x\searrow 0}y_*(x,t)=0$. This is however impossible because that would lead to $F(y_*(0,t),0,t)=0$ and thus $F(0,t) = 0\geq G(0,t)$ contradicting the assumption.\\
In the second case, observing that $F(y_*(x,t),x,t)=F(y^*(x,t),x,t)$ and simplifying we get $$\int_{y_*(x,t)}^{y^*(x,t)}(tu_0(\eta)+\eta-x)\rho_0(\eta)d\eta=0.$$
This implies
\[t u(x,t)-x=-\frac{\int_{y_*(x,t)}^{y^*(x,t)}\eta\rho_0(\eta)d\eta}{\int_{y_*(x,t)}^{y^*(x,t)}\rho_0(\eta)d\eta} \leq- y_*(x,t) .\]
Now passing to the limit as $x\searrow 0$, we obtain $tu(0+,t)\leq-y_*(0+,t) = -y^*(0,t)$ by Remark~\ref{rem:limitstaustar}. Since $G(0,t) \leq 0$ always and $F(0,t) < G(0,t)$ by assumption, we must have $y^*(0,t) > 0$ and hence
$u(0+,t) < 0$.\\
In order to prove $(2)$ first note that, for $x>0$ we have that $\tau^*(x,t)<t$. This is due to the fact that at fixed $(x,t)$ the quantity $G(\tau,x,t)$ becomes increasing for $\tau > t-x/\|u_b\|_\infty$, as seen from the definition of $G$ and remembering that we assumed $\rho_b$ and $u_b$ to be positive.
Thus the only cases in Definition~\ref{d2} that can lead to negative $u(0+,t)$ are cases with $F<G$ or $F\leq G$.
\end{proof}
\begin{theorem}\label{thm:boundary}
The pair $(\rho,u)$ as defined above solves equation \eqref{e1.1} in $\mathbb{R}_+^2$.\\
The initial conditions are satisfied in the sense that for almost all $x$ we have $\lim_{t\searrow 0} u(x,t)= u_0(x)$ and  $\rho=\partial_{x}m$ with $\lim_{t\searrow 0}m(x,t) = \int_0^x \rho_0(y) dy$.\\
The boundary condition is satisfied in regions where $F(0,t) > G(0,t)$ in the following sense: For almost all $t$ we have $\lim_{x\searrow 0} u(x,t) = u_b(t)$.\\
If in addition $u_b$ is continuously differentiable and $\rho_b$ is locally Lipschitz continuous, then $\lim_{x\searrow 0} \rho(x,t)u(x,t) = \rho_b(t)u_b(t)$.
\end{theorem}
\begin{proof}
The fact that $(\rho,u)$ is a solution is clear from the discussion in the introduction since they are derived from a solution $(m,u)$ of \eqref{system_m}. \\
To prove the validity of the initial conditions observe that
\begin{equation}
\lim_{t \searrow 0} F(y, x,t)=\int_0^y (\eta-x) \rho_0 (\eta) d \eta
\end{equation}
and
\begin{equation}
\lim_{t \searrow 0} G(x,t)=0.
\end{equation}
The assumption $\rho_0 >0$ implies that $F(x, t) < G(x,t)$ for small values of $t$.  So the initial condition holds for $u$ by the arguments of Wang \cite{Wang97}.  Since $y_{*}(x,t), y^{*}(x,t)\to x$  as $t \searrow 0$ , we get
\begin{equation*}
\lim_{t\searrow 0}m(x,t)= \int_0 ^{x} \rho_0 (y)dy.
\end{equation*}
Next we verify the boundary conditions. From our assumptions we have $F(0, t)>G(0,t)$ and by continuity, for $x$ close enough to zero, $G(x,t)< F(x,t)$.
Let $t_0$ be a Lebesgue point of $u_b$ and $\rho_b$. We start by showing that the boundary conditions for $u$ hold, i.e. $\lim_{x\searrow 0} u(x,t_0)= u_b (t_0)$.
Let $(x_n,t_0)$ be sequence converging to $(0, t_0)$. First assume that $\tau_{*}(x_n,t_0)= \tau^{*}(x_n,t_0)=\tau(x,t_0)$.  In this case, the minimum is unique and thus for $h>0$ (in fact, any $h\neq 0$) we have
\begin{equation*}
G(\tau(x_n,t_0), x_n, t_0) < G(\tau(x_n,t_0)+h x_n, x_n, t_0).
\end{equation*}
This implies
\begin{multline}
x _n\int^{\tau (x_n,t_0)+ hx_n} _ {\tau(x_n,t_0)} \rho_b (\eta)u_b(\eta)d\eta \geq \int^{\tau(x_n,t_0)+ h x_n} _
{\tau(x_n,t_0)} u^2_b (\eta)(t_0-\eta)\rho_b (\eta) d\eta\\
\geq (t_0-\tau (x_n,t_0)-hx_n)\int^{\tau(x_n,t_0)+ h x_n} _
{\tau(x_n,t)} u^2_b (\eta)\rho_b (\eta) d\eta.
\label{e5.7}
\end{multline}
Simplifying \eqref{e5.7} and using Definition \ref{d2}, we get
\begin{equation*}
\Big(\frac {1}{u(x_n, t_0)} -h \Big)\frac{\int^{\tau(x_n,t_0)+ h x_n} _
{\tau(x_n,t_0)} u^2_b (\eta)\rho_b (\eta) d\eta}
{\int^{\tau (x_n,t_0)+ hx_n} _ {\tau(x_n,t_0)} \rho_b(\eta)u_b(\eta)d\eta} \leq 1\,,
\end{equation*}
and as $n\to \infty$, we have
\begin{equation}
\limsup_{n \to \infty} \Big(\frac {1}{u(x_n, t_0)} -h \Big) u_b(t_0) \leq 1\,.
\label{e5.9}
\end{equation}
Similarly considering the inequality
 \begin{equation*}
G(\tau(x_n,t_0), y, x_n, t_0) < G(\tau(x_n,t_0)-h x_n, y, x_n, t_0),
\end{equation*}
and following the same analysis as above, we find
\begin{equation}
\liminf_{n \to \infty} \Big(\frac {1}{u(x_n, t_0)} +h \Big) u_b(t_0) \geq 1.
\label{e5.10}
\end{equation}
Since $h$ is arbitrary positive number, inequalities \eqref{e5.9} and \eqref{e5.10} result in
\begin{equation*}
\lim_{n \to \infty} \frac {1}{u(x_n, t_0)}  u_b(t_0) =1\,.
\end{equation*}
Thus we proved that  $\lim_{n \to \infty} u(x_n, t_0)=u_b (t_0)$ implying continuity of $u$ up to the boundary and validity of the boundary condition in this case.

If  $\tau_{*}(x_n,t_0)< \tau^{*}(x_n,t_0)$, then by Definition \ref{d2} we have
\begin{equation*}
u(x_n, t_0)=\dfrac {\int_{\tau_{*}(x_n,t_0)}^ {\tau^{*}(x_n,t_0)} \rho_b u^2_b}{\int_{\tau_{*}(x,t)}^ {\tau^{*}(x,t)} \rho_bu_b}
\end{equation*}
So if  $t_0$ is a Lebesgue point of $\rho_b u_b$ and $\rho_b u_b^2$, as is true for almost all $t_0$, we get $\lim_{n \to \infty} u(x_n, t_0)=u_b (t_0)$ again.\\

It remains to show that the boundary condition for $\rho$ holds. In this case, if the boundary potential $$G(\tau,x,t)=\int_{0}^{\tau} [x-u_b(\eta)(t-\eta)]\rho_b(\eta)u_b(\eta)d\eta$$ attains the minimum in $(0,\infty)$ at a point $\bar{\tau}$, then $\frac{\partial G}{\partial \tau}|_{\tau=\bar{\tau}}=0.$ That is
\[
\frac{\partial G}{\partial \tau}|_{\tau=\bar{\tau}}=\left(x-u_b(\bar{\tau})(t-\bar{\tau})\right)\rho_b(\bar{\tau})u_b(\bar{\tau})=0\,.
\]
Since $\tau_{*}(x,t), \tau^{*}(x,t)\nearrow t$  as $x \searrow 0$,  for $x$ close enough to $0$, the minimizing point lies in $(0, \infty)$.
Consider the function $f:\mathbb{R^+}\times[0,\infty)\to\mathbb{R}$ where $f$ is defined as
\[
f(x,\tau)=x-u_b(\tau)(t-\tau)\,.
\]
Then $\frac{\partial f}{\partial \tau}=-(t-\tau) u_b^{\prime}(\tau)+u_b(\tau)$, $f(0,t)=0$ and $\frac{\partial f}{\partial \tau}(0,t)>0.$ Thus by the implicit function theorem there exists a neighborhood of $(0,t)$ where $f(x,\tau)=0$ has a unique solution. That gives the unique minimizer of the boundary potential $G(\tau,x,t)$ and thus $\tau_*(x,t)$ is a continuously differentiable function of $x$ in some neighborhood of $0$. Now from the relation $u(x,t)=\frac{x}{t-\tau_*(x,t)}$, we see that $\lim_{x\searrow 0}\frac{\partial}{\partial x}\tau_*(x,t)=-\frac{1}{u_b(t)}.$
Consequently,
\begin{equation*}
\lim_{x\searrow 0}\rho(x,t)=\lim_{x\searrow 0}\frac{\partial}{\partial x}m(x,t)
=-\lim_{x\searrow 0}\rho_b(\tau_*(x,t))u_b(\tau_*(x,t))\frac{\partial}{\partial x}\tau_*(x,t)
=\rho_b(t)\,.
\end{equation*}
This completes the verification of initial and boundary condition for $(u, \rho)$.
\end{proof}
The next theorem concerns the global conservation of mass and momentum for our solutions.
\begin{theorem}
Let $\rho_0\in L^1([0, \infty))$, then the solution given in Definition~\ref{def:sol} conserves the total mass, i.e.
\[
\forall t\geq 0\colon \int_0^\infty\rho(dx,t)=\int_0^\infty \rho_0(x) dx+\int_0^t \rho_b(\eta)u_b(\eta) d\eta \,.
\]
Moreover, the total momentum is conserved \emph{but only for times without influx to the boundary}. More precisely we have for $t>0\colon$
\[
\int_0^\infty\!\!\!\! u(x,t) \rho(dx,t)\,\left\{
\begin{aligned}
&\!= &&\!\!\int_0^\infty\!\! \rho_0(x)u_0(x) dx&&\!\!\!\!+\int_0^t \rho_b(\eta)u_b^2(\eta) d\eta  &&\text{if }F(0,t)\geq G(0,t)\\
&\!\geq&&\!\!\int_0^\infty \!\!\rho_0(x)u_0(x) dx&&\!\!\!\!- \int_0^t \rho_b(\eta)u_b^2(\eta) d\eta &&\text{if }F(0,t)<G(0,t)\,.
\end{aligned}\right.
\]
\end{theorem}
\begin{proof}
First we prove the mass conservation property. We start in the situation, when $F(0,t)>G(0,t)$.
Since $F$ and $G$ are continuous functions, in a neighborhood of $(0,t)$ we have $F(x,t)>G(x,t)$ and thus $m(x,t)=-\int_{0}^{\tau_*(x,t)}u_b(\eta)\rho_b(\eta)d\eta.$ Note that $\tau_*(x,t)$ is decreasing in $x$ and $m(x,t)$ is increasing in $x$. Moreover, $m(x,t)$ is upper semicontinuous in this case as we shall prove now.\\
For this purpose consider a sequence $z_n=(x_n,t_n)\to z=(x,t).$ Then,
\begin{multline*}
\limsup_{n\to \infty} m(z_n)=-\lim_{n\to \infty}\,\sup_{n\geq k}\int_0^{\tau_*(z_n)}\rho_b(\eta)u_b(\eta)d\eta= \\
=-\lim_{n\to \infty}\int_0^{\sup\limits_{n\geq k}\tau_*(z_n)}\rho_b(\eta)u_b(\eta)d\eta=-\int_0^{\lim\limits_{n\to \infty}\sup\limits_{n\geq k}\tau_*(z_n)}\rho_b(\eta)u_b(\eta)d\eta\,.
\end{multline*}
By lower semicontinuity we have $ \tau_*(z)\leq\liminf_{n\to \infty}\tau_*(z_n)\leq \limsup_{n\to \infty}\tau_*(z_n)$. Thus we conclude
\begin{equation*}
\limsup_{n\to \infty} m(z_n)=-\int_0^{\lim\limits_{n\to \infty}\sup\limits_{n\geq k}\tau_*(z_n)}\rho_b(\eta)u_b(\eta)d\eta \leq -\int_0^{\tau_*(z)}\rho_b(\eta)u_b(\eta)d\eta=m(z)\,.
\end{equation*}
Now $m(x,t)$ is right continuous in $x$ since it is increasing and upper semicontinuous. This leads to the desired mass conservation for $F(0,t)>G(0,t)$,
\begin{equation*}
\int_0^\infty \rho(dx,t)=\int_{x\in(0,\infty)}\hspace{-0.8cm}m(dx,t)=m(\infty, t)-m(0,t)=\int_{0}^{\infty}\rho_0(x)dx+\int_{0}^t u_b(\eta)\rho_b(\eta)d\eta\,.
\end{equation*}
For the other case when $F(0,t)\leq G(0,t)$ we decompose the integral in the following manner:
\begin{equation*}
\begin{split}
\int_{0}^{\infty}\rho(dx,t)&=m(0+,t)-m(0,t)+\int_{x\in(0,\infty)}\hspace{-0.8cm}m(dx,t)=m(\infty,t)-m(0,t)=\\
&=\int_{0}^{\infty}\rho_0(x)dx+\int_{0}^{t} u_b(\eta)\rho_b(\eta)d\eta\,.
\end{split}
\end{equation*}
This finishes the proof of mass conservation.\\

Now we show momentum conservation. The momentum  $\rho. u$ is understood as
\begin{equation}
\rho(x,t)u(x,t)=\left\{\begin{aligned}
&u m_x= q_x &&\text{for } x,t>0\\
&\rho_b (t)u_b(t) &&\text{for }x=0\text{ and }F(0,t)>G(0,t)\\
&\ 0 &&\text{for }x=0\text{ and }F(0,t)< G(0,t)\\
& \rho(0,t)u(0,t)&&\text{for }x=0\text{ and } F(0,t)= G(0,t)\,.
\end{aligned}\right.
\end{equation}
Note that in the last case this is the Dirac mass at zero, multiplied by the non zero velocity expected at this point.
This is consistent with the point-wise interpretation as $\rho u$ up to the boundary if $u$ is as in Definition~\ref{d2} and $\rho$ as in Definition~\ref{def:sol}.
When $F(0,t)>G(0,t)$, by similar argument as in the proof of mass conservation one can show that $q(x,t)$ is right continuous in $x$. Hence we find
\begin{equation*}
\begin{split}
\int_{0}^{\infty} u\rho(dx,t)&=\int_{x\in (0,\infty)}q(dx,t)=q(\infty,t)-q(0,t)=\\
&=\int_{0}^{\infty}\rho_0(x)u_0(x)dx+\int_0^t\rho_b(\eta)u_b(\eta)d\eta\,.
\end{split}
\end{equation*}
Now we consider the case $F(0,t)< G(0,t)$. In this case
\begin{equation*}
\begin{split}
\int_{0}^{\infty}u\rho(dx,t)&=\int_{x\in(0,\infty)}q(dx,t)+0=q(\infty,t)-q(0+,t)=\\
&=\int_{0}^{\infty}\rho_0(x)u_0(x)dx-\int_0^{y^*(0,t)}\rho_0(\eta)u_0(\eta)d\eta\,.
\end{split}
\end{equation*}
For the last equality, Remark~\ref{rem:limitstaustar} was used.
Since we are in the case $F(0,t)<G(0,t)$, we have
\[
\int_{0}^{y_*(0,t)}(tu_0(\eta)+\eta)\rho_0(\eta)d\eta\leq -\int_0^t(t-\eta)u_b^2(\eta)\rho_b(\eta)d\eta\,,\\
\]
leading to
\begin{multline*}
0\geq-t\int_0^t u_b^2(\eta)\rho_b(\eta)d\eta-\int_0^{y_*(0,t)}\eta \rho_0(\eta)d\eta\geq\\
\geq\int_0^{y_*(0,t)}tu_0(\eta)\rho_0(\eta)d\eta-\int_0^t \eta u_b^2(\eta)\rho_b(\eta)d\eta\geq\\
\geq t\left(\int_0^{y_*(0,t)}u_0(\eta)\rho_0(\eta)d\eta-\int_0^t \rho_b(\eta)u_b^2(\eta)d\eta\right)\,.
\end{multline*}
This implies
\[
\int_0^\infty\rho(x,t)u(x,t) dx \geq \int_0^\infty \rho_0(x)u_0(x)dx-\int_0^t \rho_b(\eta)u_b^2(\eta) d\eta\,.
\]
Finally we have to consider the case $F(0,t)=G(0,t)$. First note that we have $y_*(0+,t)=y^*(0,t)$, again by Remark~\ref{rem:limitstaustar}.
Now for $F(0,t)=G(0,t)$ we have
\begin{multline*}
\int_{0}^{\infty}u\rho(dx,t)=\\
=\int_{x\in(0,\infty)}q(dx,t)+\lim_{x\searrow 0}\big(m(x,t)-m(0,t)\big)\frac{\int_{0} ^ {\tau^{*}(0,t)}\rho_b u^2_b+\int_{0} ^ {y^{*}(0,t)}\rho_0 u_0}{\int_{0} ^ {\tau^{*}(0,t)}\rho_b u_b +\int_{0} ^ {y^{*}(0,t)}\rho_0}=\\
=q(\infty,t)-q(0+,t)+\left(\int_0^{y_*(0+,t)}\rho_0 d\eta+\int_0^t\rho_b u_b d\eta\right)\frac{\int_{0} ^ {t}\rho_b u^2_b+\int_{0} ^ {y^{*}(0,t)}\rho_0 u_0}{\int_{0} ^ {t}\rho_b u_b +\int_{0} ^ {y^{*}(0,t)}\rho_0}=\\
=\int_{0}^{\infty}\rho_0(x)u_0(x)dx-\int_0^{y_*(0+,t)}\rho_0(\eta)u_0(\eta)d\eta+\int_{0} ^ {t}\rho_b u^2_bd\eta+\int_{0} ^ {y^{*}(0,t)}\rho_0 u_0d\eta=\\
=\int_{0}^{\infty}\rho_0(x)u_0(x)dx+\int_{0} ^ {t}\rho_b u^2_bd\eta\,,
\end{multline*}
completing the proof.
\end{proof}
\section{Entropy condition}
In this section we will show that the solution we constructed is an entropy solution up to the boundary.
\begin{theorem}
Let $t>0$ and $x>0$. If $x$ is a point of discontinuity of $u(\cdot,t)$, then we have
\[
u(x-,t)>u(x,t)>u(x+,t)\,.
\]
Moreover, for almost all $t>0$ we have that $u(\cdot,t)$ is either right continuous at $x = 0$ or
\[
u(0,t)>u(0+,t)\,.
\]
\end{theorem}
\begin{proof}
First let $x>0$, $t>0$.
If $(x,t)$ is a point where $F(x,t)<G(x,t)$ and $y_{*}(x,t)=y^{*}(x,t)$ then $u(x,t) = (x - y_*(x,t))/t$ is continuous at $(x,t)$.
On the other hand if we have $y_{*}(x,t) < y^{*} (x,t)$ singularities can form and this case is considered in \cite{Wang97}.\\
Now we look at points $(x,t)$, where $F(x,t)> G(x,t)$ and we  have $\tau_{*}(x,t) < \tau^{*} (x,t)$.
Using again Remark~\ref{rem:limitstaustar} we have that
$u(x-, t)= \frac{x}{t-\tau^{*}(x,t)}$ and $u(x+, t)= \frac{x}{t-\tau_{*}(x,t)}$. From
\[
G(\tau_{*}(x,t), x, t)=G(\tau^{*}(x,t), x, t)\,,
\]
we know that
\[
\int_{\tau_{*}(x,t)}^{\tau^{*}(x,t)} [x-u_b(\eta)(t-\eta)]\rho_b(\eta)u_b(\eta)d\eta=0\,.
\]
This implies
\[
 \frac{x}{t-\tau_{*}(x,t)}\int_{\tau_{*}(x,t)}^{\tau^{*}(x,t)} \rho_b(\eta)u_b(\eta) d\eta=\int_{\tau_{*}(x,t)}^{\tau^{*}(x,t)} \frac{(t-\eta)} {t-\tau_{*}(x,t)} u^2_b(\eta) \rho_b(\eta) d\eta\,,
\]
from which we conclude $u(x+, t) < u(x,t)$. The other inequality also follows easily if one divides the above equality by $(t-\tau^{*}(x,t))$.\\
The case $F(x,t)= G(x,t)$ is slightly more complicated. If $F(x,t)=G(x,t)$ in an interval then we are in the case of the rarefaction wave emanating from zero and the solution is continuous. Thus we can assume that the equality holds in an isolated point.
As already observed above, we have that $u(x-, t)=\frac{x}{t-\tau^{*}(x,t)}$.
Note that
\[
F(y^{*} (x,t), x, t)- G(\tau^{*}(x,t), x, t)<
F(y^{*} (x,t), 0, \tau^{*}(x,t))- G(\tau^{*}(x,t),0, \tau^{*}(x,t))\,,
\]
since the term on the left side is zero.
The inequality above yields
\begin{multline*}
\int_0^{y^*(x,t)}\!\!\Big(\big(t-\tau^*(x,t)\big)u_0(\eta)-x\Big)\rho_0(\eta)d\eta<\\
<\int_0^{\tau^*(x,t)}\!\!\Big(x-\big(t-\tau^*(x,t)\big)u_b(\eta)\Big)\rho_b(\eta)u_b(\eta)d\eta\,
\end{multline*}
and dividing by $t-\tau^*(x,t)$ we get $u(x,t) < \frac{x}{t-\tau^{*}(x,t)}$, using the definition of $u(x,t)$.\\
To derive the inequality for $u(x+, t)= \frac{x-y^{*}(x,t)}{t}$ we proceed in a similar way, starting from
\[
G(\tau^{*} (x,t), x, t)- F(y^{*}(x,t),x, t)<G(\tau^{*}(x,t),y^{*} (x,t), 0)- F(y^{*}(x,t),y^{*}(x,t),0)\,.
\]
This completes the proof in the interior.\\
It remains to consider the boundary $x=0$. For almost all $t>0$ such that $F(0,t)>G(0,t)$ we know from Theorem~\ref{thm:boundary} that $\lim_{x \searrow 0} u(x,t) = u_b(t)$, that is $u(\cdot,t)$ is right continuous. For $F(0,t)<G(0,t)$, Lemma~\ref{lem:vel0} implies that $u(0+,t)<0$ while $u(0,t)=0$ according to Definition~\ref{d2}.
For the only remaining case $F(0,t)=G(0,t)$, the entropy condition follows from the same argument as in the interior, since $u(0,t)$ is defined in the same manner in this situation.
\end{proof}

\section{Some explicit examples}
First we present an example that includes a rarefaction wave emerging from the origin and a Dirac delta with time-dependent mass forming from the initial data due to a downward jump in the velocity. Initial and boundary data can be read off from Figure~\ref{fig:raref}, but we repeat them for completeness. We choose the initial and boundary density to be equal to $1$ everywhere and the initial velocity to jump from $u_0=2$ to $u_0=-2$ at position $x=2$. The boundary velocity is chosen constant and smaller than the initial velocity near zero in order  to generate the rarefaction wave. The actual value is $1$. To make the construction of the solution more easy to follow we present $\mu=\min(F,G)$ in the left half of Figure~\ref{fig:raref}, indicating $F$ or $G$, whichever is smaller. The right half is the solution that we calculate from the derivatives of $\mu$ according to Lemma~\ref{lem:integrals}. The position of the Dirac delta is indicated by the bold line. The boundary conditions $\rho_b=1$, $u_b=1$ are satisfied up to $t=16/3$. At this time the $\delta$ hits the boundary and stays there, increasing its mass, as influx from the right continues. We have $m=-t$ along the $t$-axis.
\begin{figure}[ht!]
\begin{minipage}{0.5\textwidth}
\begin{tikzpicture}[scale=1.55]
%%%% axis
\draw[->] (0,0)--(3.4,0) node[anchor=north]{$x$};
\draw[->] (0,0)--(0,6.3) node[anchor=east]{$t$};
%%%% data
% in data
\draw  (1.3,0)node [anchor=north] {\parbox{2cm}{$\rho_0=1$\\$u_0=2$}};
\draw  (3.0,0)node[anchor=north] {\parbox{2cm}{$\rho_0=1$\\$u_0=-2$}};
% in ticks
\draw (2,0) node[anchor=north] {$2$};
\draw (2,-.05)--(2,0.05);
% bd data
\draw(-0.35,3.5)node[anchor=north, rotate=90]{$\rho_b=1$, $u_b=1$};
% bd ticks
\draw (-0.05,1)--(0.05,1);
\draw (0,1) node[anchor=east] {$1$};
\draw (-0.05,16/9)--(0.05,16/9);
\draw (0,16/9) node[anchor=east] {$\tfrac{16}{9}$};
\draw (-0.05,16/3)--(0.05,16/3);
\draw (0,16/3) node[anchor=east] {$\tfrac{16}{3}$};
%%%% curve
\draw (2,0)--(2,1);
\draw  (2,0.8)node[anchor=north, rotate=90]{\parbox{2cm}{$x=2$}};
\draw[scale=1, domain=1:16/9, smooth, variable=\t] plot ({-2*\t+4*sqrt(\t)},{\t});
\draw (2.67,1.3) node[anchor=south]{$x=-2t+4\sqrt{t}$};
\draw (16/9,16/9)--(0,16/3);
\draw (1.4, 3.5)node[anchor=north,rotate=-61]{$x=-\frac{1}{2}t+\frac{8}{3}$};
%%%% lines
\draw (0,0)--(2,1);
\draw (1.5,0.77)node[anchor=north,rotate=27]{$x=2t$};
\draw (0,0)--(16/9,16/9);
\draw (1,1.4)node[anchor=north,rotate=45]{$x=t$};
%%%% potential
\draw (1.25, 0.45)node[anchor=north, blue]{$F=\frac{-(2t-x)^2}{2}$};
\draw (1.1, 1.0)node[anchor=north, rotate=37, blue]{$F=G=0$};
\draw (0.7,2.4)node[anchor=north, blue]{$G=\frac{-(t-x)^2}{2}$};
\draw(1.8, 5)node[anchor=north, blue]{$F=\frac{-(2t-x)^2}{2}+8t<G=0$};
\end{tikzpicture}
\end{minipage}\hfill
\begin{minipage}{0.5\textwidth}
\begin{tikzpicture}[scale=1.55]
%%%% axis
\draw[->] (0,0)--(3.4,0) node[anchor=north]{$x$};
\draw[->] (0,0)--(0,6.3) node[anchor=east]{$t$};
%%%% data
% in data
\draw  (1.3,0)node [anchor=north] {\parbox{2cm}{$\rho_0=1$\\$u_0=2$}};
\draw  (3.0,0)node[anchor=north] {\parbox{2cm}{$\rho_0=1$\\$u_0=-2$}};
% in ticks
\draw (2,0) node[anchor=north] {$2$};
\draw (2,-.05)--(2,0.05);
% bd data
\draw(-0.35,3)node[anchor=north, rotate=90]{$\rho_b=1$, $u_b=1$};
%%%% curve
\draw [line width=0.3mm] (2,0)--(2,1);
\draw[scale=1, line width=0.3mm, domain=1:16/9, smooth, variable=\t] plot ({-2*\t+4*sqrt(\t)},{\t});
\draw [line width=0.3mm] (16/9,16/9)--(0,16/3);
\draw [line width=0.3mm] (0,16/3)--(0,6.3);
%%%% lines
\draw (0,0)--(2,1);
\draw (0,0)--(16/9,16/9);
%%%%%%%%% solution
\draw (2,5) node[anchor=north]{\parbox{3cm}{$\textcolor{blue}{m=x+2t}\\\textcolor{red}{\rho=1,\,u=-2}$}};
\draw (1.6,0.85)node[anchor=north,rotate=25]{\parbox{3cm}{\textcolor{blue}{$m=x-2t$}\\[-1mm]\textcolor{red}{$\rho=1,u=2$}}};
\draw (1.15,1.27) node[anchor=north,rotate=25]{\parbox{2cm}{\flushright\textcolor{blue}{$m=0$}\\\textcolor{red}{$\rho=0, u=\frac{x}{t}$}}};
\draw (0.2,2.8)node[anchor=north,rotate=90]{\parbox{3cm}{\textcolor{blue}{$m=x-t$}\\\textcolor{red}{$\rho=1,u=1$}}};
\draw (2,0.5)node[anchor=north, rotate=90,red]{$\rho=4t\delta$};
\draw (2.42,1.3)node[anchor=south,red]{$\rho=4\sqrt{t}\delta$};
\draw (1.3,3.5)node[anchor=north,rotate=-61,red]{$\rho=3t\delta$};
\draw (0.32,5.83)node[anchor=north,rotate=270,red]{$\rho=3t\delta$};
\draw(-0.35,4.5)node[anchor=north, rotate=90,blue]{$m=-t$};
\end{tikzpicture}
\end{minipage}
\caption{Dirac delta generated from initial data, absorbing rarefaction wave and meeting the boundary in finite time}\label{fig:raref}
\end{figure}
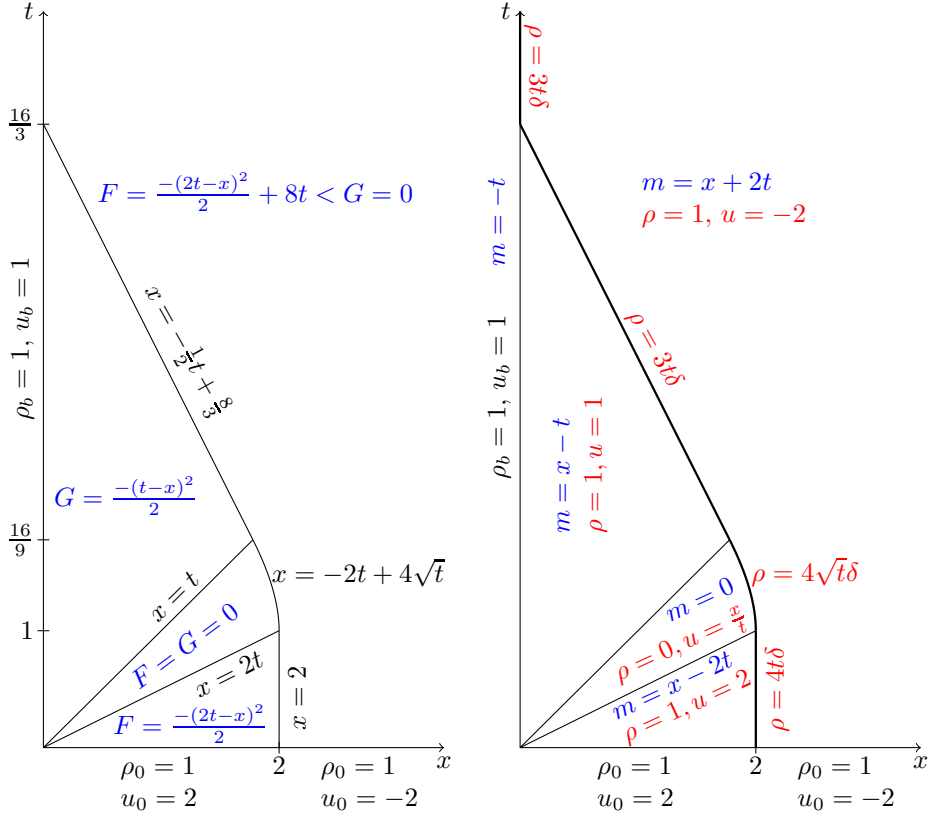

The next example, presented in Figure~\ref{bildInBd}, is interesting since it is a case where a Dirac delta at $x = 0$ forms driven by influx from the initial data. Moreover, the example is constructed in such a way that the influx from the initial manifold stops after some time while the boundary data continues to give rise to positive momentum influx. Thus the Dirac delta indeed leaves $x=0$ and moves inward when the influx of momentum from the boundary is sufficient. Note that the Dirac delta does not leave $x=0$ with zero velocity. Thus our solution can be interpreted as one for a sticky boundary. Moreover, it can be seen from this that the solution does not satisfy a semigroup property. The semigroup property could be restored by prescribing a negative momentum at the boundary in such cases, which however would lead to the momentum no longer being the product of mass and velocity. We choose not to do that for physical reasons---even if the momentum could be considered as a kind of dummy variable in this case. Note also that strictly speaking the situation depicted in Figure~\ref{bildInBd} is not covered by our theory since we chose $\rho_0=0$ for $x>2$. However, the solution is similar if $\rho_0$ is very small for $x>2$. The boundary data is again constant $\rho_b=u_b=1$ and the initial data for $x<2$ is chosen to be $\rho_0=1$ and $u_0=-2$. This leads to the formation of the Dirac delta at $x=0$ because $\rho_0u_0<\rho_bu_b$. Again we find $m=-t$ along the whole $t$-axis and we depict the position of the Dirac delta with a bold line in the right part of Figure~\ref{bildInBd}.

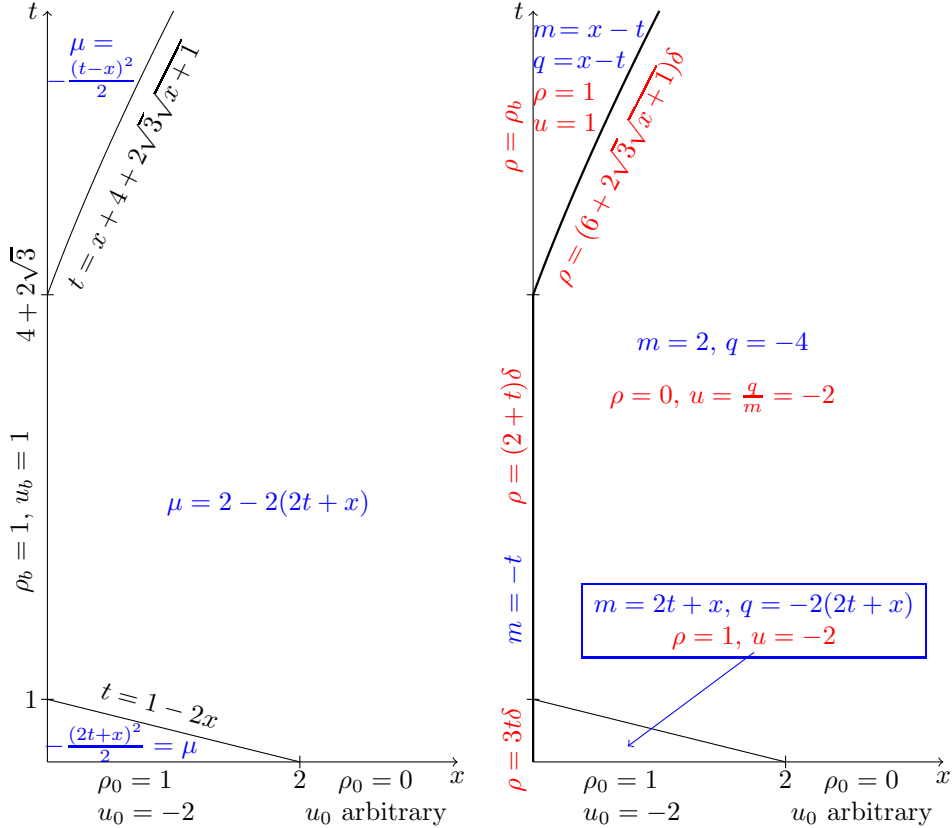
\begin{figure}[ht!]
\begin{minipage}{0.49\textwidth}
\begin{tikzpicture}[scale=0.83]
\draw[->] (0,0) -- (6.5,0) node[anchor=north] {$x$};
\draw[->] (0,0) -- (0,12) node[anchor=east] {$t$};
\draw[scale=1, domain=0:2, smooth, variable=\x] plot ({\x},{\x+4+sqrt(12*(\x+1))});
\draw (1,9.7) node[anchor=north, rotate=64] {$t=x+4+2\sqrt{3}\sqrt{x+1}$};
\draw (2,0) node[anchor=north] {\parbox{2cm}{$\rho_0=1$\\$u_0=-2$}};
\draw (5.2,0) node[anchor=north] {\parbox{2cm}{\centering$\rho_0=0$\\$u_0$ arbitrary}};
\draw (4,0) -- (0,1);
\draw (1.7,0.6) node[anchor=south,rotate=-15] {$t=1-2x$};
\node [rotate=90] at (-.4,4) {$\rho_b=1$, $u_b=1$};
\draw (4,-.1) -- (4,.1);
\draw (4,0) node[anchor=north] {$2$};
\draw (-.1,1) -- (.1,1);
\draw (0,1) node[anchor=east] {$1$};
\draw (-.1,7.464) -- (.1,7.464);
\draw (0,7.464) node[anchor=south, rotate=90] {$4+2\sqrt{3}$};
\draw (3.5,4.5) node[anchor=north,blue] {$\mu=2-2(2t+x)$} ;
\draw (1.16,0.76) node[anchor=north,blue] {$-\tfrac{(2t+x)^2}{2}=\mu$} ;
\draw (0.7,11.7) node[anchor=north,blue] {\parbox{2cm}{\centering $\mu=$\\$-\tfrac{(t-x)^2}{2}$}} ;
\end{tikzpicture}
\end{minipage}\hfill
\begin{minipage}{0.49\textwidth}
\begin{tikzpicture}[scale=0.83]
\draw[->] (0,0) -- (6.5,0) node[anchor=north] {$x$};
\draw[->] (0,0) -- (0,12) node[anchor=east] {$t$};
\draw (1,9.7) node[anchor=north, rotate=63, red] {$\rho=(6+2\sqrt{3}\sqrt{x+1})\delta$};
\draw (2,0) node[anchor=north] {\parbox{2cm}{$\rho_0=1$\\$u_0=-2$}};
\draw (5.2,0) node[anchor=north] {\parbox{2cm}{\centering$\rho_0=0$\\$u_0$ arbitrary}};
\draw (4,0) -- (0,1);
\node [rotate=90, red] at (-.3,0.2) {$\rho=3t\delta$};
\node [rotate=90, red] at (-.3,5.2) {$\rho=(2+t)\delta$};
\node [rotate=90, blue] at (-.3,2.6) {$m=-t$};
\node [rotate=90, red] at (-.3,10) {$\rho=\rho_b$};
\draw (4,-.1) -- (4,.1);
\draw (4,0) node[anchor=north] {$2$};
\draw (-.1,1) -- (.1,1);
\draw (-.1,7.464) -- (.1,7.464);
\draw (3,7) node[anchor=north, blue] {$m=2$, $q=-4$} ;
\draw (3,6.2) node[anchor=north, red] {$\rho=0$, $u=\tfrac{q}{m}=-2$} ;
\draw (3.5,3) node[anchor=north,blue] {\fbox{\parbox{4.3cm}{\centering $m=2t+x$, $q=-2(2t+x)$\\\textcolor{red}{$\rho=1$, $u=-2$}}}};
\draw [->, blue] (3.5,1.75) -- (1.5,0.25);
\draw (0.9,12) node[anchor=north, blue] {\parbox{1.5cm}{$m\!=x-t$\\$q=\!x\!-\!t$}} ;
\draw (0.9,11) node[anchor=north, red] {\parbox{1.5cm}{$\color{red}{\rho=1}$\\$u=1$}};
\draw[scale=1, domain=0:2,line width=0.3mm,smooth, variable=\x] plot ({\x},{\x+4+sqrt(12*(\x+1))});
\draw [line width=0.3mm] (0,0)--(0,7.464);
\end{tikzpicture}
\end{minipage}\\
\vspace*{-.3cm}
\caption{Left: Potential $\mu=\min\{F,G\}$. Right: Solution, calculated from potential according to Lemma~\ref{lem:integrals} and Definition~\ref{def:sol}}.\label{bildInBd}
\end{figure}
Finally in Figure~\ref{fig:twoDel} we give an example where a Dirac delta forms due to a jump up in the boundary velocity. Moreover, we include a jump down in the initial data and the two discontinuities merge into a single Dirac mass. For the values we choose the Dirac delta does not reach the boundary again. More precisely, the initial and the boundary densities are $1$ everywhere, but the initial velocity jumps down from $1$ to $-2$ at $x=2$ and the boundary velocity jumps up from $1$ to a value of $2$ at $t=1$. We give separate plots for the potentials $F$ and $G$ and include dashed auxiliary lines to indicate the regions that have to be considered in the minimization with respect to $y$ and $\tau$. The solid lines separate regions where the formulas for the potentials actually differ. The solution is plotted in the lower right of Figure~\ref{fig:twoDel}. Indeed, it satisfies the initial and boundary conditions everywhere. The bold lines again indicate the positions of the Delta masses.
\begin{figure}[ht!]
\begin{minipage}[t]{0.48\textwidth}\vspace{0cm}
\begin{tikzpicture}[scale=1.1]
\draw[->] (0,0) -- (4,0) node[anchor=north] {$x$};
\draw[->] (0,0) -- (0,4) node[anchor=east] {$t$};
\draw (2,-.05) -- (2,.05);
\draw (2,0) node[anchor=north] {$2$};
\draw (0.05,1) -- (-.05,1);
\draw (0,1) node[anchor=east] {$1$};
\draw (2,0) -- (0,4);
\draw [dashed] (2,0) -- (0,1);
\draw [dashed] (2,0) -- (4,2);
\draw (0.7,1) node[anchor=north, blue] {$-\tfrac{(t-x)^2}{2}$};
\draw (3,1) node[anchor=north, blue] {$-\tfrac{(2t+6)^2-12t}{2}$};
\draw (1,2) node[anchor=south, rotate=-62.5] {$t=-2x+4$};
\draw (3.5,0) node[anchor=north] {\parbox{2cm}{$\rho_0=1$\\$u_0=-2$}};
\draw (1.2,0) node[anchor=north] {\parbox{2cm}{$\rho_0=1$\\$u_0=1$}};
\end{tikzpicture}
\end{minipage}
\hfill
\begin{minipage}[t]{0.48\textwidth}\vspace{0cm}
\begin{tikzpicture}[scale=1.1]
\draw[->] (0,0) -- (4,0) node[anchor=north] {$x$};
\draw[->] (0,0) -- (0,4) node[anchor=east] {$t$};
\draw (3,-.05) -- (3,.05);
\draw (3,0) node[anchor=north] {$3$};
\draw (0.05,1) -- (-.05,1);
\draw (0,1) node[anchor=east] {$1$};
\draw (0.05,3) -- (-.05,3);
\draw (0,3) node[anchor=east] {$3$};
\node [rotate=90] at (-.4,2.5) {\parbox{2cm}{$\rho_b=1$\\ $u_b=2$}};
\node [rotate=90] at (-.4,0.7) {\parbox{2cm}{$\rho_b=1$\\ $u_b=1$}};
\draw (0,0) --(3,3);
\draw [dashed] (3,3) -- (4,4);
\draw [dashed] (0,1) -- (4,3);
\draw [dashed] (0,1) -- (3,4);
\draw (0,1) -- (3,3);
\draw[scale=1, domain=3:4, smooth, variable=\x] plot ({\x},{\x/2+3/4+sqrt(\x/4-3/16)});
\draw	(3,3) circle[radius=1pt, black];
\fill	(3,3) circle[radius=1pt, black];
\node [rotate=32] at (1.9,2.55) {$t=\tfrac{2}{3}x+1$};
\node at (2.7,1) {\fbox{$t=\tfrac{2x+3+\sqrt{4x-3}}{4}$}};
\draw [->] (3,1.35) -- (3.8,3.4);
\draw (1.6,3.2) node[anchor=south,blue] {$-\tfrac{(2t-x)^2+2x-6t+3}{2}$};
\draw (2.7,1.8) node[anchor=south,blue] {$0$};
\draw (0.2,0.75) node[anchor=west,blue,rotate=37] {$-\tfrac{(t-x)^2}{2}$};
\draw (0.5,0.3) node[anchor=west,rotate=45] {$t=x$};
\end{tikzpicture}
\end{minipage}\\
\begin{minipage}[t]{0.48\textwidth}
\begin{tikzpicture}[scale=1.1]
\draw[->] (0,0) -- (4,0) node[anchor=north] {$x$};
\draw[->] (0,0) -- (0,4) node[anchor=east] {$t$};
\draw (2,-.05) -- (2,.05);
\draw (2,0) node[anchor=north] {$2$};
\draw (0.05,1) -- (-.05,1);
\draw (0,1) node[anchor=east] {$1$};
\draw (0,1) -- (9/8,7/4);
\draw (2,0)--(9/8,7/4);
\draw[scale=1, domain=1.75:4, samples=100, variable=\t] plot ({3/4+17/(32*\t-12)},{\t});
\draw (1.55,3) node[anchor=north, rotate=-85] {$t=\tfrac{3x+2}{8x-6}$};
\draw (0.8,1) node[anchor=north,blue] {$-\tfrac{(t-x)^2}{2}$};
\draw (2.7,3) node[anchor=north,blue] {$-\tfrac{(2t+x)^2-12t}{2}$};
\draw (0.8,2.8) node[anchor=north,blue,rotate=-90] {$-\tfrac{(2t-x)^2+2x-6t+3}{2}$};
\draw (9/8,7/4) circle[radius=1pt];
\fill (9/8,7/4) circle[radius=1pt];
\draw (9/8,7/4) node[anchor=west] {\,\,($\tfrac{9}{8},\tfrac{7}{4}$)};
\draw (0.75,0.05) -- (0.75,-.05);
\draw (0.75,0) node[anchor=north] {$\tfrac{3}{4}$};
\end{tikzpicture}
\end{minipage}
\hfill
\begin{minipage}[t]{0.48\textwidth}\centering
\begin{tikzpicture}[scale=1.1]
\draw[->] (0,0) -- (4,0) node[anchor=north] {$x$};
\draw[->] (0,0) -- (0,4) node[anchor=east] {$t$};
\draw (2,-.05) -- (2,.05);
\draw (2,0) node[anchor=north] {$2$};
\draw (0.05,1) -- (-.05,1);
\draw (0,1) node[anchor=east] {$1$};
\draw [line width=0.3mm](0,1) -- (9/8,7/4);
\draw [line width=0.3mm](2,0)--(9/8,7/4);
\draw[scale=1, line width=0.3mm, domain=1.75:4, samples=100, variable=\t] plot ({3/4+17/(32*\t-12)},{\t});
\draw (1.2,1) node[anchor=north,red] {\parbox{2cm}{$\rho=1$\\$u=1$}};
\draw (2.9,3) node[anchor=north,red] {\parbox{2cm}{$\rho=1$\\$u=-2$}};
\draw (1,3.2) node[anchor=north,red] {\parbox{2cm}{$\rho=1$\\$u=2$}};
\draw (9/8,7/4) circle[radius=1pt];
\fill (9/8,7/4) circle[radius=1pt];
\draw (0.75,0.05) -- (0.75,-.05);
\draw (0.75,0) node[anchor=north] {$\tfrac{3}{4}$};
\draw (1.5,1.4) node[anchor=west, rotate=-65,red] {$\rho=3t\delta$};
\draw (0,1.3) node[anchor=west,rotate=30,red] {$(t-1)\delta$};
\draw (1.5,3) node[anchor=north, rotate=-85,red] {$\rho=(4t-1)\delta$};
\end{tikzpicture}
\end{minipage}\\
\caption{Upper left: $F$; upper right: $G$; lower left: $\mu=\min\{F,G\}$; lower right: Solution satisfying initial and boundary conditions. The delta approaches $x=3/4$ for large time.}\label{fig:twoDel}
\end{figure}
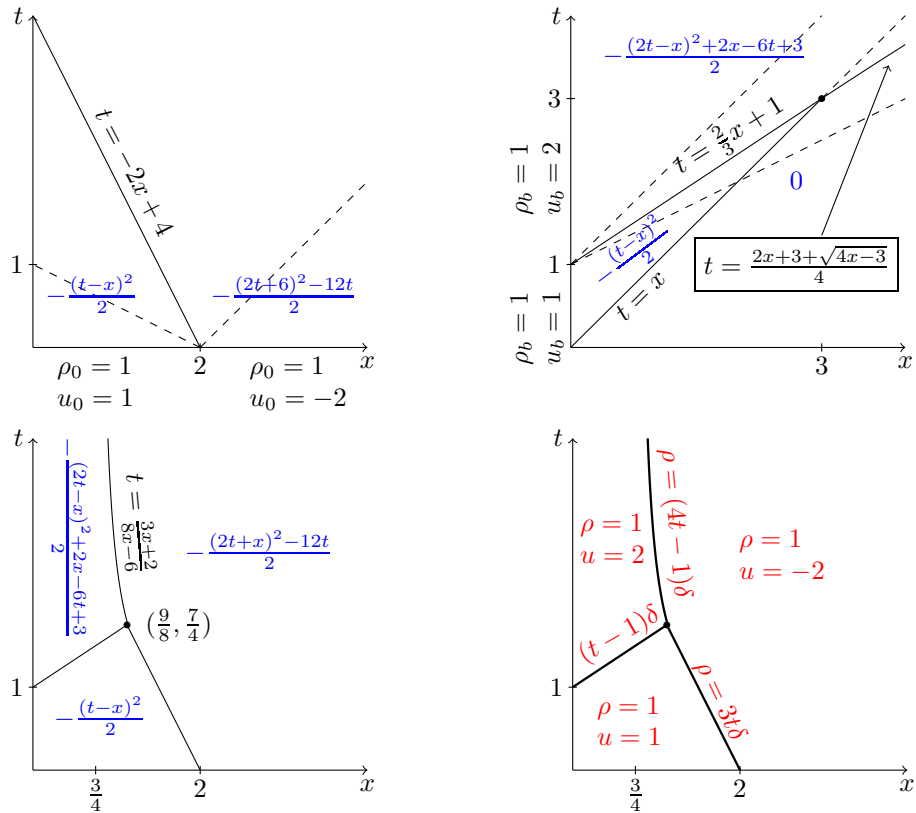


\begin{thebibliography}{10}
\providecommand{\url}[1]{{#1}}
\providecommand{\urlprefix}{URL }
\expandafter\ifx\csname urlstyle\endcsname\relax
  \providecommand{\doi}[1]{DOI~\discretionary{}{}{}#1}\else
  \providecommand{\doi}{DOI~\discretionary{}{}{}\begingroup
  \urlstyle{rm}\Url}\fi

\bibitem{bar}
Bardos, C., {L}e Roux, A.Y., N\'{e}d\'{e}lec, J.C.: First order quasilinear
  equations with boundary conditions.
\newblock Comm. Partial Differential Equations \textbf{4}(9), 1017--1034
  (1979).
\newblock \urlprefix\url{https://doi.org/10.1080/03605307908820117}

\bibitem{Bouchut94}
Bouchut, F.: On zero pressure gas dynamics.
\newblock In: B.~Perthame (ed.) Advances in kinetic theory and computing,
  \emph{Series on Advances in Mathematics for Applied Sciences}, vol.~22, pp.
  171--190. World Scientific Publishing Co., Inc., River Edge, NJ (1994).
\newblock Selected papers

\bibitem{BouchutJames98}
Bouchut, F., James, F.: One-dimensional transport equations with discontinuous
  coefficients.
\newblock Nonlinear Anal. \textbf{32}(7), 891--933 (1998).
\newblock \urlprefix\url{https://doi.org/10.1016/S0362-546X(97)00536-1}

\bibitem{BouchutJames99}
Bouchut, F., James, F.: Duality solutions for pressureless gases, monotone
  scalar conservation laws, and uniqueness.
\newblock Comm. Partial Differential Equations \textbf{24}(11-12), 2173--2189
  (1999).
\newblock \urlprefix\url{https://doi.org/10.1080/03605309908821498}

\bibitem{Boudin00}
Boudin, L.: A solution with bounded expansion rate to the model of viscous
  pressureless gases.
\newblock SIAM J. Math. Anal. \textbf{32}(1), 172--193 (2000).
\newblock \urlprefix\url{https://doi.org/10.1137/S0036141098346840}

\bibitem{BrenierGrenier98}
Brenier, Y., Grenier, E.: Sticky particles and scalar conservation laws.
\newblock SIAM J. Numer. Anal. \textbf{35}(6), 2317--2328 (1998).
\newblock \urlprefix\url{https://doi.org/10.1137/S0036142997317353}

\bibitem{d3}
Ding, X., Wang, Z.: Existence and uniqueness of discontinuous solutions defined
  by {L}ebesgue-{S}tieltjes integral.
\newblock Sci. China Ser. A \textbf{39}(8), 807--819 (1996)

\bibitem{ERykovSinai96}
E, W., Rykov, Y.G., Sinai, Y.G.: Generalized variational principles, global
  weak solutions and behavior with random initial data for systems of
  conservation laws arising in adhesion particle dynamics.
\newblock Comm. Math. Phys. \textbf{177}(2), 349--380 (1996).
\newblock \urlprefix\url{http://projecteuclid.org/euclid.cmp/1104286332}

\bibitem{Huang05}
Huang, F.: Weak solution to pressureless type system.
\newblock Comm. Partial Differential Equations \textbf{30}(1-3), 283--304
  (2005).
\newblock \urlprefix\url{https://doi.org/10.1081/PDE-200050026}

\bibitem{Wang01}
Huang, F., Wang, Z.: Well posedness for pressureless flow.
\newblock Comm. Math. Phys. \textbf{222}(1), 117--146 (2001).
\newblock \urlprefix\url{https://doi.org/10.1007/s002200100506}

\bibitem{ma2}
Joseph, K.T., Sahoo, M.R.: Boundary {R}iemann problem for the one-dimensional
  adhesion model.
\newblock Can. Appl. Math. Q. \textbf{19}(1), 19--41 (2011)

\bibitem{JosephSahoo12}
Joseph, K.T., Sahoo, M.R.: A system of generalized {B}urgers' equations with
  boundary conditions.
\newblock Appl. Math. E-Notes \textbf{12}, 102--109 (2012)

\bibitem{j3}
Joseph, K.T., Veerappa~Gowda, G.D.: Explicit formula for the solution of convex
  conservation laws with boundary condition.
\newblock Duke Math. J. \textbf{62}(2), 401--416 (1991).
\newblock \urlprefix\url{https://doi.org/10.1215/S0012-7094-91-06216-2}

\bibitem{Le1}
LeFloch, P.: Explicit formula for scalar nonlinear conservation laws with
  boundary condition.
\newblock Math. Methods Appl. Sci. \textbf{10}(3), 265--287 (1988).
\newblock \urlprefix\url{https://doi.org/10.1002/mma.1670100305}

\bibitem{NNOS17}
Nedeljkov, M., Neumann, L., Oberguggenberger, M., Sahoo, M.R.: Radially
  symmetric shadow wave solutions to the system of pressureless gas dynamics in
  arbitrary dimensions.
\newblock Nonlinear Anal. \textbf{163}, 104--126 (2017).
\newblock \urlprefix\url{https://doi.org/10.1016/j.na.2017.07.006}

\bibitem{Wang97}
Wang, Z., Ding, X.: Uniqueness of generalized solution for the {C}auchy problem
  of transportation equations.
\newblock Acta Math. Sci. (English Ed.) \textbf{17}(3), 341--352 (1997).
\newblock \urlprefix\url{https://doi.org/10.1016/S0252-9602(17)30852-4}

\bibitem{WHD97}
Wang, Z., Huang, F., Ding, X.: On the {C}auchy problem of transportation
  equations.
\newblock Acta Math. Appl. Sinica (English Ser.) \textbf{13}(2), 113--122
  (1997).
\newblock \urlprefix\url{https://doi.org/10.1007/BF02015132}

\end{thebibliography}
\end{document}